\documentclass[12pt, reqno]{amsart}

\usepackage{amsmath, amsthm, amssymb, stmaryrd}
\usepackage{hyperref}
\usepackage{enumerate}
\usepackage{url}

\usepackage{xcolor}

\usepackage[centering, margin={1in, 0.5in}, includeheadfoot]{geometry}

\usepackage{fancyvrb}

\newtheorem{thm}{Theorem}[section]
\newtheorem{lemma}[thm]{Lemma}

\newtheorem{cor}[thm]{Corollary}
\newtheorem{conj}[thm]{Conjecture}

\theoremstyle{definition}

\theoremstyle{remark}
\newtheorem{remark}[thm]{Remark}

\numberwithin{equation}{section}

\newcommand{\mmod}[1]{\,\,\text{mod}\,\,#1}

\def\alp{{\alpha}} 
\def\bet{{\beta}}  
\def\gam{{\gamma}} 
\def\del{{\delta}} \def\Del{{\Delta}}

\def\tet{{\theta}}  
\def\kap{{\kappa}}
\def\lam{{\lambda}} \def\Lam{{\Lambda}}

\def\sig{{\sigma}}

\def\ome{{\omega}}  
\def\eps{\varepsilon}

\def\le{\leqslant} \def\ge{\geqslant}

\def \sig{{\sigma}}

\def \bA {\mathbb A}

\def \bC {\mathbb C}

\def \bF {\mathbb F}

\def \bN {\mathbb N}

\def \bQ {\mathbb Q}
\def \bR {\mathbb R}
\def \bZ {\mathbb Z}

\def \be {\mathbf e}

\def \bt {\mathbf t}

\def \bx {\mathbf x}
\def \bw {\mathbf w}

\def \balp {\boldsymbol{\alp}}

\def \fg {\mathfrak g}

\def \fA {\mathfrak A}
\def \fB {\mathfrak B}

\def \fF {\mathfrak F}
\def \fG {\mathfrak G}
\def \fH {\mathfrak H}

\def \fL {\mathfrak L}

\def \fP {\mathfrak P}

\def \cC {\mathcal C}
\def \cD {\mathcal D}
\def \cE {\mathcal E}
\def \cF {\mathcal F}
\def \cG {\mathcal G}
\def \cH {\mathcal H}

\def \cL {\mathcal L}

\def \cN {\mathcal N}

\def \cP {\mathcal P}

\def \cY {\mathcal Y}
\def \cZ {\mathcal Z}

\def \rank {\mathrm{rank}}

\def \ord {\mathrm{ord}}

\def \deg {\mathrm{deg}}

\def \Gal {{\mathrm{Gal}}}
\def\Res{{\mathrm{Res}}}

\begin{document}
\title[Towards van der Waerden]{Towards van der Waerden's conjecture}
\author{Sam Chow \and Rainer Dietmann}
\address{Mathematics Institute, Zeeman Building, University of Warwick, Coventry CV4 7AL, United Kingdom}
\email{Sam.Chow@warwick.ac.uk}
\address{Department of Mathematics, Royal Holloway, University of London\\
Egham TW20 0EX, United Kingdom}
\email{Rainer.Dietmann@rhul.ac.uk}

\begin{abstract} How often is a quintic polynomial solvable by radicals? We establish that the number of such polynomials, monic and irreducible with integer coefficients in $[-H,H]$, is $O(H^{3.91})$. More generally, we show that if $n \ge 3$ and $n \notin \{ 7, 8, 10 \}$ then there are $O(H^{n-1.017})$ monic, irreducible polynomials of degree $n$ with integer coefficients in $[-H,H]$ and Galois group not containing $A_n$. Save for the alternating group and degrees $7,8,10$, this establishes a 1936 conjecture of van der Waerden.
\end{abstract}

\maketitle

\section{Introduction}

\subsection{Counting solvable quintics}

By the Abel--Ruffini theorem, there are integer quintic polynomials that cannot be solved by radicals. But how often is a quintic polynomial solvable by radicals? 

Let $H$ be a large, positive real number. In this article, one objective is to count monic, irreducible quintic polynomials
\begin{equation} \label{fdef}
f(X) = X^5 + a X^4 + b X^3 + cX^2 + dX + e
\end{equation}
that are solvable by radicals, where $a,b,c,d,e \in \bZ \cap [-H,H]$, denoting their number by $N(H)$. It follows from a more general result \cite[Theorem 1]{Die2012} that $N(H) \ll_\eps H^{\eps + 25/6}$, and we are able to improve upon this. 
\begin{thm} \label{QuinticBound} We have
\[
N(H) \ll_\eps H^{\eps + 7/2 + 1/\sqrt 6} \le H^{3.91}.
\]
\end{thm}
\noindent We will see that $H^4$ is a significant threshold.

\bigskip

The Galois group of a separable polynomial is the automorphism group of its splitting field \cite[\S 6.3]{Cox}. An irreducible polynomial is solvable by radicals if and only if its Galois group is a solvable group. In this way, the question posed can be viewed through the lens of \emph{enumerative Galois theory}, a topic in arithmetic statistics concerned with the frequency of Galois groups of polynomials or number fields.

Enumerative Galois theory for polynomials has a long history. It follows from Hilbert's irreducibility theorem \cite{Hilbert} that the number $\cE_n(H)$ of monic, irreducible, non-$S_n$ integer polynomials with coefficients in $[-H,H]$ is $o(H^n)$ as $H \to \infty$. In 1936, van der Waerden showed that
\[
\cE_n(H) \ll_n H^{n - 1/(6(n-2) \log \log H)},
\]
and conjectured that $\cE_n(H) = o(R_n(H))$ as $H \to \infty$, where $R_n(H)$ counts reducible polynomials \cite{vdW1936}. We know that if $n \ge 3$ then
\[
R_n(H) = c_n H^{n-1} + O(H^{n-2}(\log H)^2),
\]
where $c_n$ is a positive constant, see \cite{Che1963} and \cite[Appendix B]{CD2020}. Thus, van der Waerden's conjecture can be equivalently stated as follows.

\begin{conj} [van der Waerden 1936]
For $n \ge 3$, we have
\[
\cE_n(H) = o(H^{n-1}) \qquad \text{as } H \to \infty.
\]
\end{conj}

\begin{remark} Since the initial release of this manuscript, Bhargava has resolved a weak form of van der Waerden's conjecture \cite{Bhargava}, namely that $\cE_n(H) \ll_n H^{n-1}$. The present article concerns a stronger form of the conjecture, that $\cE_n(H) = o(H^{n-1})$. Van der Waerden \cite{vdW1936} wrote that ``Es scheint n\"amlich, da{\ss} die irreduziblen Polynome mit Affekt noch erheblich seltener sind als die reduziblen Polynome'', speculating that irreducible non-$S_n$ polynomials are substantially rarer than reducible polynomials.
\end{remark}

In \cite{Die2013}, the second named author showed that if $n \ge 3$ then
\[
\cE_n(H) \ll_{n,\eps} H^{\eps + n - 2 + \sqrt 2},
\]
breaking a record previously held by van der Waerden \cite{vdW1936}, Knobloch \cite{Kno1956}, Gallagher \cite{Gal1972} and Zywina \cite{Zyw2010}. Recently, we settled the cubic and quartic cases of van der Waerden's conjecture.

\begin{thm} [From \cite{CD2020}] \label{CubicQuartic} Van der Waerden's conjecture holds for $n=3$ and $n=4$. Moreover, we have
\[
H \ll \cE_3(H) \ll_\eps H^{\eps + 3/2}, \qquad H^2 (\log H)^2 \ll \cE_4(H) \ll H^{\eps +5/2 + 1/\sqrt 6} \ll H^{2.91}
\]
\end{thm}

\noindent S. Xiao \cite{Xiao2020} was able to prove a stronger bound in the cubic case, namely $\cE_3(H) \ll H (\log H)^2$.

\bigskip

More generally, how often does each group occur as the Galois group of a polynomial of a fixed degree $n \ge 3$? For $G \le S_n$, let us write $N_{G,n} = N_{G,n}(H)$ for the number of monic, irreducible, integer polynomials, with coefficients bounded by $H$ in absolute value, whose Galois group is conjugate to $G$. The second named author showed in \cite{Die2012} that
\begin{equation} \label{general}
N_{G,n} \ll_{n,\eps} H^{\eps + n-1 + 1/[S_n : G]}.
\end{equation}

The Galois group $G_f$ of an irreducible polynomial $f$ acts transitively on its roots \cite[Proposition 6.3.7]{Cox}, that is, it is a transitive subgroup of $S_n$. This greatly limits the number of possibilities for the conjugacy class of $G_f$. For example, in the case $n=5$ of quintic polynomials, the transitive subgroups of $S_5$ are: 
\begin{itemize}
\item $S_5$, the full symmetric group on five elements, which has order 120 and is insolvable;
\item $A_5$, the alternating group, which has order 60 and is insolvable;
\item $\mathrm{AGL}(1,\bF_5)$, the general affine group, which has order 20 and is solvable;
\item $D_5$, the dihedral group of order 10, which is solvable; and
\item $C_5$, the cyclic group of order 5, which is solvable.
\end{itemize}
As $S_5$ and $A_5$ are the only transitive subgroups of $S_5$ that are insolvable, Theorem \ref{QuinticBound} makes progress towards the quintic case of van der Waerden's conjecture. It implies that
\begin{equation} \label{SolvableGroups}
N_{\mathrm{AGL}(1,\bF_5),5} + N_{D_5,5} + N_{C_5,5} 
= N(H) \ll H^{3.91} = o(H^4)
\end{equation}
so, for the quintic case of van der Waerden's conjecture, only $A_5$ remains. This is to say that we have the following consequence of Theorem \ref{QuinticBound}.

\begin{cor} \label{QuinticSmallGalois} If $N_{A_5,5} = o(H^4)$ then van der Waerden's conjecture holds for $n=5$.
\end{cor}

For $s,t \in \bZ$, the de Moivre quintic
\[
X^5 + 5s X^3 + 5s^2 X + t
\]
has Galois group $\mathrm{AGL}(1,\bF_5)$ whenever it is irreducible, see \cite[Example 13.2.10]{Cox}. It now follows from Eisenstein's criterion that
\[
N_{\mathrm{AGL}(1,\bF_5), 5} \gg H^{3/2}.
\]
As $\mathrm{AGL}(1,\bF_5)$ is solvable, we thus obtain the following lower bound to complement the upper bound in Theorem \ref{QuinticBound}.

\begin{cor} 
We have
\[
N(H) \gg H^{3/2}.
\]
\end{cor}

\subsection{Sextic and higher-degree polynomials}

Save for the alternating group, we are able to prove van der Waerden's conjecture in any degree except for $7,8,10$. Here is our main theorem.

\begin{thm} [Main Theorem] \label{MainThm} 
If $n \ge 3$ and $n \notin \{7, 8, 10\}$ then there are $O(H^{n-1.017})$ monic, irreducible polynomials of degree $n$ with integer coefficients in $[-H,H]$ and Galois group not containing $A_n$.
\end{thm}

\begin{cor} \label{MainCor}
Let $n \ge 3$ with $n \notin \{7,8,10\}$, and suppose $N_{A_n,n} = o(H^{n-1})$ as $H \to \infty$. Then van der Waerden's conjecture holds in degree $n$.
\end{cor}

When $n = 9$ or $n \ge 11$, we reach Theorem \ref{MainThm} via the index of $G$.

\begin{thm} \label{HigherThm} Let $n = 9$ or $n \ge 11$, let $G$ be a transitive subgroup of $S_n$ that does not contain $A_n$, and put $d = [S_n:G]$. Then $d \ge 240$, and
\[
N_{G,n} \ll_{n,\eps} H^{n+\eps-3/2+3d^{-1/3}}
\le H^{n-1.017}.
\]
\end{thm}

In the sextic case, we have the following estimates.

\begin{thm} \label{SexticThm} Let $G$ be a transitive subgroup of $S_6$ that does not contain $A_6$. Then
\[
N_{G,6} \ll_\eps H^{\eps+9/2+1/\sqrt{6}} \le H^{4.91}.
\]
Moreover, if $G$ is solvable then
\[
N_{G,6} \ll_\eps H^{\eps+9/2+1/\sqrt{10}} \le H^{4.82}.
\]
\end{thm}

\begin{cor} There are $O(H^{4.82})$ monic, irreducible, solvable sextic polynomials with integer coefficients in $[-H,H]$.
\end{cor}

Note that Theorem \ref{MainThm} follows from Theorems \ref{QuinticBound}, \ref{CubicQuartic}, \ref{HigherThm} and \ref{SexticThm}.
We can also handle proper subgroups of the alternating group, exploiting the additional information that the discriminant is a square.

\begin{thm} \label{EvenGroupThm} Let $n \ge 3$ with
\begin{equation} \label{erratum}
n \notin \{ u^2, u^2+1 \} \qquad (u \in \bN \text{ odd}).
\end{equation}
Let $G$ be a proper, transitive subgroup of $A_n$, and put $d = [S_n:G]$. Then $d \ge 6$ and
\[
N_{G,n} \ll_{n,\eps} H^{\eps+ n-3/2+1/\sqrt d} \le H^{n-1.09}.
\]
\end{thm}

In particular, the theorem applies when $n \in \{7,8\}$. For each $n \in \{7,8,10\}$, these being the exceptional degrees in Theorem \ref{MainThm}, this leaves just one conjugacy class of $G$ not containing $A_n$ for which we do not know that $N_G = o(H^{n-1})$. These are called 7T4, 8T47 and 10T43, see the tables of Butler and McKay \cite{BM1983}. It is probable that the condition \eqref{erratum} can be removed through further reasoning, see \cite[\S 6]{DOS}.

\subsection{Methods}

We approach the problem from the perspective of diophantine equations. Recall that cubic and quartic polynomials were treated in \cite{CD2020}.

\subsubsection{Quintic polynomials}
\label{QuinticBrief}

We begin by introducing a standard invariant $\Del = \Del_f$ of a monic quintic polynomial $f$, the \emph{discriminant}, and a standard auxiliary polynomial $\tet = \tet_f$, the \emph{sextic resolvent}. If $f$ is solvable, then $\tet$ has a rational root \cite[Corollary 13.2.11]{Cox}, which must in fact be an integer as a consequence of $f$ being monic. By introducing this root $y$ as an additional variable, we obtain a diophantine equation. If $f$ has coefficients in $[-H,H]$, then we will see that $y \ll H^2$. With the coefficients as in \eqref{fdef} we show that, for generic $a,b,c$, the vanishing of $\tet(a,b,c,d,e,y)$ defines an absolutely irreducible surface in $\bA^3_{d,e,y}$ containing no rational lines. We need to count integer points on this surface, bounded by a rectangle of dimensions $2H, 2H, O(H^2)$, uniformly in $a,b,c$. We achieve this using the \emph{determinant method}. Pioneered by Bombieri and Pila who studied planar curves \cite{BP1989}, this has seen a number of developments over the years, and enables us to cover the overwhelming majority of these integer points by a relatively small number of curves, which eventually turns out to be decisive. Specifically, we use Browning's version \cite[Lemma 1]{Bro2011}, which relies on deep work of Salberger, see Theorem \ref{BrowningThm}.

\subsubsection{Sextic polynomials}

We approach sextic polynomials in much the same way as we approach quintic polynomials. The resolvents are much more complicated, to the extent that they are inconvenient to print in expanded form. We computed and manipulated them using the software \emph{Mathematica} \cite{Mathematica}. The case of insolvable sextics is more demanding both technically and computationally; one of the steps is to show that our surfaces generically lack certain low-degree curves.

\subsubsection{Higher-degree polynomials}

In this case we have general resolvents $\Phi_G$ from \cite{Die2012} and \cite{CD2017}. If the Galois group is $G$, then $\Phi_G$ has an integer root, which we introduce as an extra variable to obtain a diophantine equation in $y$ and the coefficients of $f$. Fixing all but three of the coefficients of $f$ generically furnishes an absolutely irreducible threefold. We can cover its integer points up to height $H$ by a relatively small number of surfaces, using a determinant method result of Heath-Brown's \cite[Theorem 15]{Cetraro}. A result of Pila's \cite{Pila1995} enables us to efficiently count integer points on an irreducible component of such a surface, unless it happens to be linear. The latter scenario concerns a two-parameter family of polynomials, whose Galois theory can be understood over the two-parameter function field using the framework developed by Uchida \cite{Uch1970}, J. H. Smith \cite{Smi1977} and S. D. Cohen \cite{Coh1972, Coh1980}. Refining the separable resolvent theory of \cite{CD2017}, we are led to the problem of counting integer points on an irreducible, high-degree surface in a lopsided box. We complete the proof by establishing a lopsided version of Pila's theorem \cite{Pila1995} in a special case. To achieve the latter we first establish an almost-uniform, quantitative version of Hilbert's irreducibility theorem, refining a case of \cite[Theorem 2.1]{Coh1981}.

For the lopsided version version of Pila's theorem and for the almost-uniform, quantitative version of Hilbert's irreducibility theorem, see \S \ref{AuxSection}. For Galois theory over two-parameter function fields, we find that the Galois group of 
\[
f(X) = G_0(X) + aG_1(X) + bG_2(X)
\]
over $\bC(a,b)$ is $S_n$ in the following two new cases, assuming that $f$ is irreducible over $\bQ(a,b)$, for any $n \ge 6$, any $\alp, \bet, \gam \in \bQ$ and any $a_1,\ldots,a_{n-3} \in \bZ$ for which we have \eqref{weird}:
\begin{enumerate}[(i)]
\item
\[
G_0(X) = X^n + a_1X^{n-1} + \cdots + a_{n-3}X^3 - \alp,
\quad G_1(X) = X^2 - \bet,
\quad G_2(X) = X - \gam
\]
\item
\[
G_0(X) = X^n + a_1 X^{n-1} + \cdots + a_{n-3}X^3 - \alp,
\quad G_1(X) = X^2 - \bet,
\quad G_2(X) = 1.
\]
\end{enumerate}

\subsection{Variants}

Enumerative Galois theory for number fields is another thriving area of research, and one that is closely related to enumerative Galois theory for polynomials \cite{LLOT, LOT, Won2005}. There are some subtle differences between the two types of problem, for instance Bhargava~\cite{Bha2005} famously showed that a positive proportion of quartic fields fail to have full Galois group $S_4$. 
One can also try to relax the condition of the polynomial being monic, equivalently counting binary forms. Different heights can be considered, for example heights that respect linear transformations, or that are more compatible with other arithmetic statistics such as number fields, class groups, and ranks of elliptic curves \cite{BS2015, BSTTTZ, Won2005, Xiao2020}. More general probability distributions for the coefficients  have been considered in \cite{PhamXu2021}. Recently there have been substantial developments in the problem with small coefficients and large degree \cite{ BKK2021, BK2020, BV2019}. There have also been some nice results on the distribution of Galois groups of characteristic polynomials of matrices \cite{Ebe2021, JKZ2012, Riv2008}.

\subsection{Organisation}

We discuss finer aspects of the methods in \S \ref{FurtherMethods}.
We prove Theorem \ref{QuinticBound} in \S \ref{QuinticSection}. In \S \ref{SeparableResolvents}, we discuss the role of separable resolvents, and extend the existing work \cite{CD2017, Die2012} on this topic. Theorem \ref{SexticThm} is proved in \S \ref{SexticSection}. In \S \ref{AuxSection}, we establish an almost-uniform version of Hilbert's irreducibility theorem and a lopsided version of Pila's theorem, as well as making preparations for Galois theory over two-parameter function fields. Then, in \S \ref{HigherSection}, we establish Theorem \ref{HigherThm}. Finally, in \S \ref{EvenGroups}, we prove Theorem \ref{EvenGroupThm}. 

\subsection{Notation} We adopt the convention that $\eps$ denotes an arbitrarily small positive constant, whose value is allowed to change between occurrences. We use the Vinogradov and Bachmann--Landau notations throughout, the implicit constants being allowed to depend on $\eps$. Throughout $H$ denotes a positive real number, sufficiently large in terms of $\eps$. If $F$ is a polynomial with complex coefficients then $|F|$ is the greatest absolute value of its coefficients.

\subsection{Funding and acknowledgements} SC was supported by EPSRC Fellowship Grant EP/S00226X/2, and by the Swedish Research Council under grant no. 2016-06596. We thank an anonymous referee for detailed feedback, and Bijay Bhatta for drawing our attention to a significant typo.

\section{Further discussion of the methods}
\label{FurtherMethods}

\subsection{Resolvents}

Let
\[
f(X) = X^n + a_1 X^{n-1} + \cdots + a_n \in \bZ[X].
\]
A \emph{resolvent} is an auxiliary polynomial $\Phi(Y)$ whose factorisation type provides information about the Galois group of $f$. The resolvents that we use in this manuscript have a specific property; to a transitive subgroup $G$ of $S_n$, we associate $\Phi = \Phi_G$ such that if $G_f \le G$ then $\Phi$ has a root $y \in \bZ$. To show that the condition $G_f \le G$ is unlikely, we count integer solutions to
\begin{equation} \label{hypersurface}
\Phi(y;a_1,\ldots,a_n) = 0,
\end{equation}
subject to size constraints. For us, these constraints always have the shape
\[
|a_1|, \ldots, |a_n| \le H,
\qquad
y \ll H^{O(1)}.
\]

\subsection{The determinant method}
\label{DeterminantMethod}

The next step is to fix all but two or three of the variables $a_1,\ldots,a_n$ in some generic way, so that the resulting surface or threefold is absolutely irreducible. This simplifies our problem substantially, but on the other hand its coefficients are now of size $H^{O(1)}$ instead of being constant. The \emph{determinant method} is a versatile approach to counting integer points on varieties, with good uniformity in the coefficients. This makes it an ideal weapon for our approach.

The archetypal application of the determinant method is the following theorem \cite{BP1989}.

\begin{thm} [Bombieri--Pila 1989] \label{BPthm} Let $C$ be an absolutely irreducible algebraic curve of degree $d \ge 2$ in $\bR^2$. Then the number of integer points in $C \cap [-H,H]^2$ is $O_{d,\eps}(H^{\eps+1/d})$.
\end{thm}

\noindent We use this and its lopsided generalisation, as stated in \cite[Lemma 8]{Die2012}. This was essentially given by Browning and Heath-Brown \cite{BHB2005}, however a short argument has been incorporated in order to relax the absolute irreducibility requirement to irreducibility over $\bQ$.

\begin{thm} [Lopsided Bombieri--Pila] \label{LopsidedBP}
Let $F \in \bZ[x_1,x_2]$ be irreducible over $\bQ$ and have degree $d \in \bN$. Further, let $B_1, B_2 \ge 1$, and define
\[
N(F; B_1, B_2) = \: \# \{\bx \in \bZ^2:
F(\bx) = 0, \: |x_i| \le B_i \: (1 \le i \le 2) \}.
\]
Put
\[
T = \max \left \{ B_1^{e_1} B_2^{e_2} \right \},
\]
where the maximum is taken over all $(e_1,e_2) \in \bZ_{\ge 0}^2$ for which $x_1^{e_1} x_2^{e_2}$
occurs in $F(\bx)$ with non-zero coefficient. Then
\[
N(F; B_1, B_2) \ll_{d, \eps} T^\eps \exp \left( \frac{\log B_1 \cdot \log B_2}{\log T} \right).
\]
\end{thm}

\noindent Pila \cite{Pila1995} subsequently generalised Theorem \ref{BPthm}.

\begin{thm} [Pila 1995, special case] \label{PilaThm} Let $N \ge 2$ be an integer, and let $V$ be an irreducible affine hypersurface in $\bR^N$ of degree $d \ge 2$. Then the number of integer points in $V \cap [-H,H]^N$ is $O_{N,d,\eps} (H^{N-2+\eps+1/d})$.
\end{thm}

\noindent We establish a lopsided variant of this, see Theorem \ref{LopsidedPila}, which we use in our proof of Theorem \ref{HigherThm}.

A key idea in the determinant method is to cover the integer points on a variety by lower-dimensional varieties. A famous result of this type is Heath-Brown's \cite[Theorem 15]{Cetraro}, which is also lopsided.

\begin{thm} [Heath-Brown 2006] \label{HBthm} Let $N \ge 2$ be an integer. Suppose $F(x_1, \ldots,x_N) \in \bZ[\bx]$ defines an absolutely irreducible hypersurface of degree $d \in \bN$. Let $B_1,\ldots,B_N \ge 1$. Put
\[
T = \max \left \{ \prod_{i \le N} B_i^{e_i} \right \},
\]
where the maximum is taken over all $(e_1,...,e_N) \in \bZ_{\ge 0}^N$ for which $x_1^{e_1} \cdots x_N^{e_N}$
occurs in $F(\bx)$ with non-zero coefficient. Then there exist $g_1, \ldots, g_J \in \bZ[\bx]$, each coprime to $F$ and of degree $O_{N,d,\eps}(1)$, where
\[
J \ll_{N,d,\eps} T^\eps \exp \left \{
(N-1) \left(\frac{\prod \log B_i}{\log T} \right)^{1/(N-1)}
\right\} (\log |F|)^{2N-3},
\]
such that if $\bx \in ([-B_1,B_1] \cap \bZ) \times \cdots \times ([-B_N, B_N] \cap \bZ)$ and $F(\bx) = 0$ then $g_j(\bx) = 0$ for some $j$.
\end{thm}

\noindent Note that $T$ will be larger, and the bound on $J$ stronger, if $d$ is large. This is the case for our resolvents when $n$ is large, and so we apply Theorem \ref{HBthm} with $N = 4$ in our proof of Theorem \ref{HigherThm}. We will find that $|F| \ll H^{O(1)}$ in our applications, and it is essential that we have a bound of this flavour.

In the case $N = 3$ of surfaces, the factor of $N-1$ can be removed at the expense of introducing a small exceptional set, using Salberger's rather sophisticated adelic machinery. The mild dependence on $|F|$ can also be removed in this case. Such a result was formulated by Browning in \cite[Lemma 1]{Bro2011}.

\begin{thm} [Browning 2011] \label{BrowningThm}
Suppose $F(x_1, x_2, x_3) \in \bZ[\bx]$ defines an absolutely irreducible surface of degree $d \in \bN$, and let $B_1,B_2,B_3 \ge 1$. Put
\[
T = \max \left \{ B_1^{e_1} B_2^{e_2} B_3^{e_3} \right \},
\]
where the maximum is taken over all $(e_1, e_2, e_3) \in \bZ_{\ge 0}^3$ for which $x_1^{e_1} x_2^{e_2} x_3^{e_3}$
occurs in $F(\bx)$ with non-zero coefficient, and also put
\[
V_3 = \exp \left \{
\left(\frac{\log B_1 \cdot \log B_2 \cdot \log B_3}{\log T} \right)^{1/2}
\right\}.
\]
Then there exist $g_1, \ldots, g_J \in \bZ[\bx]$ and $Z \subset \bZ^3$, with 
\[
J \ll_{d,\eps} T^\eps V_3,
\qquad
|Z| \ll_{d,\eps} T^\eps V_3^2,
\]
such that the following hold:
\begin{enumerate}[(i)]
\item Each $g_j$ is coprime to $F$ and has degree $O_{d,\eps}(1)$;
\item If $\bx \in ([-B_1,B_1] \cap \bZ) \times ([-B_2,B_2] \cap \bZ) \times ([-B_3, B_3] \cap \bZ) \setminus Z$ and $F(\bx) = 0$ then $g_j(\bx) = 0$ for some $j$.
\end{enumerate}
\end{thm}

\noindent It remains an open problem to achieve such a refinement for larger $N$. In particular, our methods would be more effective if we had it in the case $N = 4$ of threefolds.

As mentioned, uniformity in the coefficients is a key strength of the determinant method. We close this subsection with a brief discussion of the level of uniformity required for our arguments to succeed. Observe that a factor of $(\log |F|)^{2N-3}$ appears in Theorem \ref{HBthm}. Consequently, we require that
\[
\log |F| \ll H^\eps.
\]
As discussed, we will see that we have $|F| \ll H^{O(1)}$ in practice, which is even stronger.

\subsection{Covering integer points on surfaces by curves}
\label{LowerStrategy}

We use an assortment of resolvents in the cases $n = 5$ and $n = 6$. Importantly, these are computationally tractable, in that we can manipulate the coefficients using mathematical software. To fix ideas, we focus our discussion on the case $n = 5$ of solvable quintics. We write our polynomial as
\[
f(X) = X^5 + a X^4 + bX^3 + cX^2 + dX + e,
\]
and we recall the sextic resolvent $\tet$ from \S \ref{QuinticBrief}. We need to count integer solutions to
\[
\tet(y; a,b,c,d,e) = 0
\]
with $a,b,c,d,e \ll H$ and $y \ll H^2$. The idea is to fix $a,b,c$ such that the vanishing of
\[
g(d,e,y) =  \tet(y; a,b,c,d,e)
\]
defines an absolutely irreducible affine surface $Y_{a,b,c}$. We can show that this condition is generic, in a strong quantitative sense.

Theorem \ref{BrowningThm} enables us to cover the integer points of controlled height on $Y_{a,b,c}$ by a small number of curves, up to a small set of exceptional points. Such a curve has the form
\[
g(d,e,y) = h(d,e,y) = 0,
\]
for some polynomial $h$ that is coprime to $g$. By elimination theory, we can work with an irreducible planar curve $\cC_1$ given by $F(d,e) = 0$.

\subsubsection{Rational lines}

At this stage there is a marked dichotomy. If $\cC_1$ is a line, then it can contain roughly $H$ integer points up to height $H$, and the basic strategy of counting points on $\cC_1$ is inadequate. If, however, the curve $\cC_1$ is non-linear, then Theorem \ref{LopsidedBP} would assure us that it contains $O(H^{\eps+1/2})$ integer points up to height $H$. 

If $\cC_1$ is a line, then we can use its equation to obtain a planar curve $\cC_2$ in $y$ and one of the other variables, and we may assume that $\cC_2$ is irreducible. If $\cC_2$ is non-linear, then Theorem \ref{LopsidedBP} delivers a satisfactory count. If $\cC_2$ is linear, then the linear equations defining $\cC_1$ and $\cC_2$ entail that $Y_{a,b,c}$ contains a rational line.
For generic $a,b,c$, we can show that the surface $Y_{a,b,c}$ does not contain a rational line, by comparing coefficients. However, this step is computer-assisted, so we are unable to carry it out if $n$ is large. Thus, Theorem \ref{HigherThm} requires an alternate strategy.

\subsection{Covering integer points on threefolds by surfaces}
\label{HigherStrategy}

For larger values of $n$, we use a more general resolvent $\Phi$, and we need to count integer solutions to 
\[
\Phi(y; a_1, \ldots, a_n) = 0
\]
with $a_1, \ldots, a_n \ll H$ and $y \ll H^C$, where $C = O_n(1)$. The idea is to fix $a_1,\ldots, a_{n-3}$ such that the equation
\[
\Phi(y; a_1,\ldots,a_{n-3},a,b,c) = 0
\]
defines an absolutely irreducible affine threefold $\cY = \cY_{a_1,\ldots,a_{n-3}}$. We can show that the latter condition is generic. Theorem \ref{HBthm} enables us to cover the integer points of controlled height on $\cY$ by a small number of surfaces. 

We initially follow the approach of \S \ref{LowerStrategy}, giving
$\cF(a,b,c) = 0$
for some irreducible polynomial $\cF$ of bounded degree. Theorem \ref{PilaThm} yields
\[
\# \{ (a,b,c) \in (\bZ \cap [-H,H])^3 \} \ll H^{\eps + 3/2},
\]
unless $\cF$ is linear. If $\cF$ is linear then this count can be as large as $H^2$, so a further idea is needed.

\subsubsection{Two-parameter function fields}

Using the linear equation $\cF = 0$, we can write our original polynomial in the form
\[
f(X) = G_0(X) + a G_1(X) + b G_2(X),
\]
for some $a,b \in \bZ$ and some polynomials $G_0, G_1, G_2$. Using the aforementioned framework of Uchida, Smith, and Cohen, we can show that
\[
\Gal(f, \bC(a,b)) = S_n.
\]
This requires a two-parameter linear family of polynomials, which is why we initially fixed all but three of the variables $a_1,\ldots,a_n$, instead of all but two. We then have
\[
\Gal(f, \bQ(a,b)) = S_n.
\]
Using resolvents, we can finally deduce that a generic polynomial in this two-parameter linear family has Galois group $S_n$.

\subsection{Comparing the two approaches}

As discussed, the method of \S \ref{LowerStrategy} fails for higher degrees because of its machine dependence. To see why the strategy of \S \ref{HigherStrategy} fails for lower degrees, we return to the points made in \S \ref{DeterminantMethod}. Heath-Brown's Theorem \ref{HBthm} incurs an additional factor of $N-1 = 3$ in the exponent, and will only prevail for us if the index $d = [S_n: G]$ is large. If $G$ does not contain $A_n$ and $n \ge 12$, for instance, then we know from \cite[Lemma 3]{Die2012} that
$d \ge 462$. This is why the method of \S \ref{HigherStrategy} is more effective when the degree is large.

\section{Counting solvable quintics}
\label{QuinticSection}

\subsection{The sextic resolvent}
\label{SexticResolvent}

In this subsection, we introduce the basic objects and definitions needed for Theorem \ref{QuinticBound}. The \emph{discriminant} of $f$ given by \eqref{fdef} is an octic polynomial in $a,b,c,d,e$, explicitly given by
{\scriptsize
\begin{align*}
\Del &= 256 a^5 e^3 - 192 a^4 b d e^2 - 128 a^4 c^2 e^2 + 144 a^4 c d^2 e - 27 a^4 d^4 + 144 a^3 b^2 c e^2  - 6 a^3 b^2 d^2 e  - 80 a^3 b c^2 d e + 18 a^3 b c d^3 \\ & \quad
- 1600 a^3 b e^3  + 16 a^3 c^4 e - 4 a^3 c^3 d^2 + 160 a^3 c d e^2 - 36 a^3 d^3 e  - 27 a^2 b^4 e^2 + 18 a^2 b^3 c d e   - 4 a^2 b^3 d^3 - 4 a^2 b^2 c^3 e \\ & \quad
+ a^2 b^2 c^2 d^2  + 1020 a^2 b^2 d e^2 + 560 a^2 b c^2 e^2   - 746 a^2 b c d^2 e + 144 a^2 b d^4  + 24 a^2 c^3 d e - 6 a^2 c^2 d^3 + 2000 a^2 c e^3 - 50 a^2 d^2 e^2 \\& \quad
- 630 a b^3 c e^2  + 24 a b^3 d^2 e  + 356 a b^2 c^2 d e - 80 a b^2 c d^3  + 2250 a b^2 e^3 - 72 a b c^4 e + 18 a b c^3 d^2 - 2050 a b c d e^2  + 160 a b d^3 e \\& \quad
- 900 a c^3 e^2 + 1020 a c^2 d^2 e  - 192 a c d^4 - 2500 a d e^3 + 108 b^5 e^2 - 72 b^4 c d e  + 16 b^4 d^3  + 16 b^3 c^3 e 
- 4 b^3 c^2 d^2  - 900 b^3 d e^2 \\& \quad
+ 825 b^2 c^2 e^2 + 560 b^2 c d^2 e - 128 b^2 d^4 - 630 b c^3 d e + 144 b c^2 d^3 - 3750 b c e^3 + 2000 b d^2 e^2 + 108 c^5 e - 27 c^4 d^2  \\& \quad 
+ 2250 c^2 d e^2 - 1600 c d^3 e + 256 d^5 + 3125 e^4.
\end{align*}
}\noindent
Denote by $\tet$ the \emph{sextic resolvent} \cite[Chapter 13]{Cox}, explicitly given by
\[
\tet(y) = \tet(a,b,c,d,e,y) = (y^3 + B_2 y^2 + B_4 y + B_6)^2 - 2^{10} \Del y,
\]
where
\begin{align*}
B_2 &= 8 ac - 3b^2 - 20d, \\
B_4 &= 3b^4 - 16 ab^2 c 
+ 16 a^2 c^2 + 16 bc^2 + 16 a^2 b d - 8b^2 d \\
 &\qquad  - 112acd + 240 d^2 - 64 a^3 e + 240 abe - 400ce,
\end{align*}
and
\begin{align*}
B_6 &= 8ab^4c - b^6 - 16 a^2 b^2 c^2 - 16 b^3 c^2 + 64 ab c^3 - 64 c^4 - 16 a^2 b^3 d + 28 b^4 d + 64 a^3 b c d \\
& \quad - 112 ab^2 cd 
- 128 a^2 c^2 d  + 224 b c^2 d - 64 a^4 d^2 + 224 a^2 b d^2  - 176 b^2 d^2 - 64 acd^2 + 320 d^3  \\
& \quad + 48 ab^3 e  - 192 a^2 bce  - 80b^2 c e + 640 a c^2 e + 384 a^3 d e  - 640 abde - 1600 cde\\ &\quad  - 1600 a^2 e^2 + 4000 be^2.
\end{align*}
This is an integer polynomial that is sextic in $y$. If 
\[
a,b,c,d,e \in \bZ,
\qquad |a|,|b|,|c|,|d|,|e| \le H,
\]
and $G_f$ is solvable, then $\tet$ has an integer root $y$. Then $y \ll H^2$, for if $|y| > CH^2$ for a large constant $C$ then
\[
y^6 \ll \Del |y| \ll H^8 |y|,
\]
contradicting that $|y| > CH^2$. 

To estimate the number of integers solutions to $\tet(a,b,c,d,e,y) = 0$ with $a,b,c,d,e \ll H$ and $y \ll H^2$, we begin by choosing $a,b,c \ll H$. The vanishing of $\tet(a,b,c,d,e,y)$ then cuts out an affine surface $Y_{a,b,c}$. A heuristic application of the determinant method, cf. \cite[\S 5]{CD2020}, suggests that we should eventually conclude that there are $O(H^{\eps + 7/2 + 1/\sqrt 6})$ solutions. In order to formalise this, there are two types of triples $(a,b,c)$ to consider. The first is the \emph{degenerate} type, where $Y_{a,b,c}$ either contains a rational line or is not absolutely irreducible. The second is the \emph{non-degenerate} type, where $Y_{a,b,c}$ is absolutely irreducible and contains no rational lines. 

\subsection{Non-degenerate triples} 
\label{QuinticNonDegenerate}

In this subsection, we show that there are $O(H^{\eps + 7/2 + 1/\sqrt 6})$ integer quintuples $(a,b,c,d,e) \in  [-H,H]^5$
such that $(a,b,c)$ is non-degenerate and $f$ as given by \eqref{fdef} is irreducible and solvable. First choose a non-degenerate triple $(a,b,c)$ in one of $O(H^3)$ possible ways. Let 
\[
g(d,e,y) = y^6 + \sum_{i=0}^5 c_i(d,e) y^i
\]
be $\tet(a,b,c,d,e,y)$ specialised to values of $a,b,c$, so that $Y_{a,b,c}$ is defined by the vanishing of $g$. With $C$ a large, positive constant, invoking the determinant method enables us to essentially cover 
\[
Y_{a,b,c} \cap ([-H,H]\cap \bZ)^2 \times ([-CH^2,CH^2] \cap \bZ)
\]
by a controlled number of curves, with estimates that are uniform in $a,b,c$. By Theorem \ref{BrowningThm}, there exist $g_1, \ldots, g_J \in \bZ[d,e,y]$, and $Z = Z_{a,b,c} \subset \bZ^3$, with $J \ll H^{\eps + 1/\sqrt 6}$ and $|Z| \ll H^{\eps + 2/\sqrt6}$, such that the following hold:
\begin{enumerate}[(i)]
\item Each $g_j$ is coprime to $g$ and has degree $O(1)$;
\item If $(d,e,y) \in Y_{a,b,c} \cap ([-H,H]\cap \bZ)^2 \times ([-CH^2,CH^2] \cap \bZ) \setminus Z$ then
\begin{equation} \label{TwoEq}
g(d,e,y) = g_j(d,e,y) = 0
\end{equation}
for some $j$.
\end{enumerate}
The total contribution from $(a,b,c)$ non-degenerate and $(d,e,y) \in Z_{a,b,c}$ is 
\[
O(H^{3+\eps + 2/\sqrt 6})
\ll H^{\eps + 7/2 + 1/\sqrt 6},
\]
so it remains to count solutions to
\eqref{TwoEq}, given $j$.

If $\deg_y(g_j) = 0$ then let $F(d,e) = g_j(d,e,y)$. Otherwise, let $F(d,e)$ be the resultant of $g$ and $g_j$ in the variable $y$. By \cite[Chapter 3, \S 6, Proposition 3]{CLO2015}, applied with $k$ as the fraction field of $\bZ[d,e]$, this is a non-zero element of $\bZ[d,e]$. By \cite[Chapter 3, \S 6, Proposition 5]{CLO2015}, we have $F(d,e) = 0$ for any solution $(d,e,y)$ to \eqref{TwoEq}.

Observe that $F(d,e) = 0$ if and only if we have $\cF(d,e) = 0$ for some irreducible divisor $\cF(d,e) \in \bQ[d,e]$ of $F(d,e)$. If $\cF(d,e)$ is non-linear, then Theorem \ref{LopsidedBP} gives
\[
\# \{ (d,e) \in (\bZ \cap [-H,H])^2: \cF(d,e) = 0 \} \ll H^{\eps + 1/2}.
\]
Then $y$ is determined by $g(d,e,y) = 0$ in at most six ways, so the number of solutions $(d,e,y)$ is $O(H^{\eps+1/2})$, and the contribution from this case is $O(H^{\eps+7/2+1/\sqrt{6}})$.

Suppose instead that $\cF(d,e)$ is linear. Then
\[
\alp d + \bet e + \gam = 0,
\]
for some $(\alp, \bet, \gam) \in (\bQ^2 \setminus \{ (0,0) \}) \times \bQ$. If $\bet \ne 0$ then substitute $e = -\bet^{-1} (\alp d + \gam)$ into $g(d,e,y) = 0$, giving
\begin{equation} \label{Pline}
y^6 + \sum_{i=0}^5 P_i(d) y^i = 0,
\end{equation}
where 
\[
P_i(d) = c_i(d, -\bet^{-1} (\alp d + \gam)) \in \bQ[d] \qquad (0 \le i \le 5).
\]
Factorise the left hand side of \eqref{Pline} over $\bQ$, and let $\cP(d,y) \in \bQ[d,y]$ be an irreducible factor. Note that $\cP(d,y)$ is non-linear, for if it were linear then
\[
\cP(d,y) = \cF(d,e) = 0
\]
would define a rational linear subvariety of $Y_{a,b,c}$, of dimension greater than or equal to 1, contradicting the non-degeneracy of $(a,b,c)$. Now Theorem \ref{LopsidedBP} yields
\[
\# \{ (d,y) \in \bZ^2 \cap [-H,H] \times [-CH^2,CH^2]: \cP(d,y) = 0 \} \ll H^{\eps+2/D},
\]
where $C$ is a large, positive constant and $D = \max\{ \deg_d(\cP), 2\deg_y(\cP) \} \ge 2$.

If $D \ge 4$ then the bound is adequate, so we may suppose that $D \in \{ 2, 3 \}$, and in particular $\deg_y (\cP) \le 1$. As $\cP(d,y)$ divides
\[
y^6 + \sum_{i=0}^5 P_i(d) y^i,
\]
we infer that $\cP(d,y)$ is a rational multiple of $y - F(d)$, where $F(d) \in \bQ[d]$ has degree 2 or 3. Thus, we may assume that
\[
\cP(d,y) = y - F(d).
\]
Next, write
\[
(y - F(d)) \left (y^5 + \sum_{i = 0}^4  f_i(d) y^i \right)
= y^6 + \sum_{i=0}^5 P_i(d) y^i.
\]
For $i = 2,4,6$, let 
\[
b_i(d) = B_i(a,b,c,d, -\bet^{-1} (\alp d + \gam)),
\]
and also write
\[
\cD(d) = 2^{10} \Del (a,b,c,d, -\bet^{-1} (\alp d + \gam)).
\]
Now
\[
(y - F(d)) \left (y^5 + \sum_{i = 0}^4  f_i(d) y^i \right)
= (y^3 + b_2(d) y^2 + b_4(d) y + b_6(d))^2 - \cD(d) y.
\]

Equating coefficients in $y$ yields
\begin{align*}
f_4(d) - F(d) &= 2b_2(d) \\
f_3(d) - f_4(d) F(d) &= b_2(d)^2 + 2b_4(d) \\
f_2(d) - f_3(d) F(d) &= 2b_2(d)b_4(d) + 2b_6(d) \\
f_1(d) - f_2(d) F(d) &= b_4(d)^2 + 2 b_2(d) b_6(d) \\
f_0(d) - f_1(d) F(d) &= 2 b_4(d) b_6(d) - \cD(d) \\
-f_0(d) F(d) &= b_6(d)^2.
\end{align*}
As $\deg(F) = D \ge 2$ and $\deg (b_2) =1$, we must have $\deg (f_4) = D$. As $\deg(b_4) = 2$, we must then have $\deg(f_3) = 2D$. As $\deg(b_6) = 3$, we must then have $\deg(f_2) = 3D$, whereupon $\deg(f_1) = 4D$.  Finally, as $\deg (\cD) = 5$, we obtain $\deg(f_0) = 5D$, contradicting the final equation. 

\bigskip

If instead $\bet = 0$, then we substitute $d = -\gam/\alp$ into $g(d,e,y) = 0$ and apply similar reasoning. This time
\begin{align*}
b_i(e) &= B_i(a,b,c,-\gam/\alp, e) \qquad (i=2,4,6), \\
\cD(e) &= 2^{10} \Del(a,b,c,-\gam/\alp, e),
\end{align*}
and 
\[
\deg(b_2) \le 0, \qquad \deg(b_4) \le 1, \qquad \deg(b_6) \le 2,
\qquad \deg(\cD) = 4.
\]
Then the argument of the $\bet \ne 0$ case carries through. 

\bigskip

The upshot is that the contribution from $(a,b,c)$ non-degenerate and $(d,e,y) \in Y_{a,b,c} \setminus Z_{a,b,c}$ is $O(H^3 J H^{\eps + 1/2}) \ll H^{2\eps + 7/2 + 1/\sqrt 6}.$ Having also considered the case $(d,e,y) \in Z_{a,b,c}$, we conclude that the total contribution from non-degenerate triples $(a,b,c)$ is $O( H^{\eps + 7/2 + 1/\sqrt 6}).$

\subsection{Degenerate triples} 
\label{QuinticDegenerate}

We now complete the proof of Theorem \ref{QuinticBound} by considering degenerate triples.

\subsubsection{Lines on the surface}

The idea is to show that, for generic $a,b,c$, the surface $Y_{a,b,c}$ contains no rational lines. A rational line on $Y_{a,b,c}$ has one of the following parametrisations:

\begin{enumerate}[I.]
\item $\cL = \{ (0, \cE, \cY) + t (1,E,Y) : t \in \bQ \}$, for some $\cE, \cY, E, Y \in \bQ$;
\item $\cL = \{ (\cD, 0 , \cY) + t (0,1,Y) : t \in \bQ \}$, for some $\cD, \cY, Y \in \bQ$;
\item $\cL = \{ (\cD, \cE, 0) + t (0,0,1) : t \in \bQ \}$, for some $\cD, \cE \in \bQ$.
\end{enumerate}

Let $g(d,e,y)$ be $\tet(a,b,c,d,e, y)$ specialised to the chosen values of $a,b,c$, so that $Y_{a,b,c}$ is defined by the vanishing of $g$. In each case, we substitute the parametrisation of the line into $g(d,e,y) = 0$ and expand it as a polynomial in $t$, and this has to be the zero polynomial. 

In Case I, the polynomial has degree at most $6$ in $t$, with sextic coefficient
\[
(Y-4)^4 (Y^2-24Y+400),
\]
so $Y=4$. After substituting this, the polynomial has degree at most $5$ in $t$, with quintic coefficient
\[
-20480(a - 5E)^4,
\]
so $E=a/5$. After substituting this, the polynomial has degree at most $4$ in $t$, with quartic coefficient
\[
\frac{262144 a^8}{125} \: - \: \frac{524288 a^6 b}{25} + \frac{393216 a^4 b^2}{5} - 
131072 a^2 b^3 + 81920 b^4.
\]
Hence, there are $O(H^2)$ integer triples $(a,b,c) \in [-H,H]^3$ such that $Y_{a,b,c}$ contains a rational line of Type I. 

In Case II, the polynomial has degree at most $6$ in $t$, with sextic coefficient $Y^6$, so we must have $Y=0$. Substituting this, and using the `Eliminate' command in \emph{Mathematica}, we see that the vanishing of the quartic, cubic, and quadratic coefficients determines an equation $P(a,b,c) = 0$, where $P \in \bZ[a,b,c]$ is non-trivial of degree $O(1)$. Consequently, there are $O(H^2)$ integer triples $(a,b,c) \in [-H,H]^3$ such that $Y_{a,b,c}$ contains a rational line of Type II.

In Case III, the polynomial is monic and sextic, so Case III cannot occur. Having considered all cases, we conclude that there are $O(H^2)$ integer triples $(a,b,c) \in [-H,H]^3$ such that $Y_{a,b,c}$ contains a rational line.

\subsubsection{Absolute irreducibility}

The idea is to show that $Y_{a,b,c}$ is absolutely irreducible for generic $a,b,c$. It is easy to verify that $g(d,e,y) = \tet(a,b,c,d,e,y)$ is sextic in $d,e,y$, irrespective of $a,b,c$. It then follows from classical theory \cite[Chapter V, Theorem 2A]{Schmidt} that there exist integer polynomials $p_1(a,b,c),\ldots,p_s(a,b,c) \in \bQ[a,b,c]$ of degree $O(1)$ such that if $a,b,c \in \bQ$ then
\[
Y_{a,b,c} \text{ is reducible over } \overline{\bQ} \Leftrightarrow
p_i(a,b,c) = 0 \: (1 \le i \le s).
\]
The next lemma shows that $p_i$ is non-zero for some $i$.

\begin{lemma} $Y_{0,0,0}$ is absolutely irreducible.
\end{lemma}
 
\begin{proof} The surface $Y_{0,0,0} \subset \bA^3_{d,e,y}$ is cut out by the vanishing of
\[
g(d,e,y) = (y^3-20dy^2 + 240d^2 y + 320d^3)^2 - 1024 (256 d^5 + 3125 e^4)y,
\]
which is monic in $y$, so it suffices to show that $g(1,e,y)$ is absolutely irreducible. 
\emph{Mathematica} assures us that
\begin{equation*}
  g(1,e,y)=102400 - 108544 y - 3200000 e^4 y + 44800
  y^2 - 8960 y^3 + 880 y^4 - 
 40 y^5 + y^6
\end{equation*}
is irreducible over $\bQ$. Moreover, its
Newton polygon is
the convex hull of $(0,0)$, $(0,1)$, $(4,1)$, $(0,2)$, $(0,3)$,$(0,4)$,$(0,5)$,$(0,6)$, which has vertices $(0,0)$, $(0,6)$, $(4,1)$. The greatest
common divisor of $0,0,0,6,4,1$ is $1$, so by \cite[Proposition $3$]{BCG2010} we conclude
that $g(1,e,y)$ is absolutely irreducible, as required.
\end{proof}

Consequently, the triples $(a,b,c)$ such that $Y_{a,b,c}$ is not absolutely irreducible are all zeros of a fixed non-zero polynomial of bounded degree. In particular, there are $O(H^2)$ such integer triples $(a,b,c) \in [-H,H]$.

\subsubsection{The sextic resolvent (reprise)}

In the previous two subsubsections, we showed that there are $O(H^2)$ degenerate triples $(a,b,c) \in [-H,H]$. Let us now fix such a triple. Let $F(d,e;X)$ be $f(a,b,c,d,e;X)$ specialised to our fixed values of the coefficients $a,b,c$. By \cite[Lemma 2]{Die2012}, there are $O(1)$ values of $d \in \bZ$ for which $F(d,e;X) \in \bQ(e)[X]$ has non-$S_5$ Galois group over $\bQ(e)$, uniformly in $a,b,c$. The contribution to $N(H)$ from these $O(1)$ special choices of $d$ is $O(H^3)$, since there are $O(H)$ possibilities for $e$. 

For the other $O(H)$ specialisations of $d$, the Galois group of $F(d,e;X) \in \bQ(e)[X]$ is $S_5$. Having specialised $d$ as well as $a,b,c$, we return to examine the sextic resolvent.

\begin{lemma} Let $a,b,c,d \in \bZ$, and suppose that
\[
\Gal(X^5 + aX^4 + bX^3 + cX^2 + dX + e, \bQ(e)) = S_5.
\]
Then
\[
h(e,y) := g(d,e,y) \in \bZ[e,y]
\]
is irreducible over $\bQ$.
\end{lemma}

\begin{proof} It follows from a generalisation of Hilbert's irreducibility theorem \cite[Theorem 2.1]{Coh1981} that for some $e_0 \in \bZ$ we have
\[
\Gal(X^5 + aX^4 + bX^3 + cX^2 + dX + e_0, \bQ) = S_5.
\]
Now, by \cite[Proposition 13.2.7]{Cox}, the sextic resolvent $h(e_0,y) \in \bZ[y]$ is irreducible over $\bQ$.

Let us write
\[
h(e,y) = y^6 + b_1(e) y^5 + \cdots + b_6(e),
\]
where $b_1(e), \ldots, b_6(e) \in \bZ[e]$, and suppose for a contradiction that $h$ is reducible over $\bQ$. Then
\[
h(e,y) = (y^k + c_1(e) y^{k-1} + \cdots + c_k(e))
(y^{6-k} + d_1(e) y^{5-k} + \cdots + d_{6-k}(e)) \in \bQ[e,y],
\]
for some $k \in \{1,2,3,4,5\}$ and some
$c_1(e),\ldots,c_k(e),d_1(e),\ldots,d_{6-k}(e) \in \bQ[e]$. Now
\[
h(e_0,y) = (y^k + c_1(e_0) y^{k-1} + \cdots + c_k(e_0))
(y^{6-k} + d_1(e_0) y^{5-k} + \cdots + d_{6-k}(e_0)),
\]
contradicting the irreducibility of the univariate polynomial $h(e_0,y)$ over $\bQ$.
\end{proof}

As $h(e,y)$ contains the monomial $y^6$, Theorem \ref{LopsidedBP} now gives
\[
\# \{ (e,y) \in \bZ^2: |e| \le H, \: y \ll H^2, \: h(e,y) = 0 \} \ll H^{\eps + 1/6}.
\]
These specialisations of $d$ therefore contribute $O(H^{\eps + 19/6})$ to $N(H)$.

We conclude that there are $O(H^{\eps + 19/6})$ quintuples $(a,b,c,d,e) \in (\bZ \cap [-H,H])^5$, with $(a,b,c)$ degenerate, such that $f(a,b,c,d,e;X) \in \bZ[X]$ is solvable. This, coupled with the conclusion of \S \ref{QuinticNonDegenerate}, completes the proof of Theorem \ref{QuinticBound}.

\section{Separable and irreducible resolvents}
\label{SeparableResolvents}

Let $n \ge 3$ be an integer, and let $G$ be a proper subgroup of $S_n$. In the sequel, we fix a total ordering of $\bC$, so that the Galois group of an irreducible polynomial
\begin{equation} \label{fdefGeneral}
f(X) = X^n + a_1 X^{n-1} + \cdots + a_{n-1}X + a_n
= (X - \alp_1) \cdots (X - \alp_n) \in \bZ[X]
\end{equation}
is a well-defined subgroup of $S_n$, without any notion of equivalence.

A resolvent is an auxiliary polynomial whose factor type reveals information about the Galois group of the original polynomial. As a consequence of \cite[Proposition 13.2.7]{Cox}, the sextic resolvent that we saw in the previous section has the property of always being separable, as long as $f$ is irreducible. It is easy to show that the quadratic resolvent $y^2 - \Del$, where $\Del$ is the discrimant, is a resolvent for the alternating group that also has this property. So too does the cubic resolvent of a quartic polynomial, which was used in \cite{CD2020}, since a quartic polynomial has the same discriminant as its cubic resolvent. In general, however, it is not known whether an always separable resolvent exists for an arbitrary permutation group $G$.

Separability is a highly desirable property for a resolvent to have, see \cite[\S 13.3]{Cox}. It can be used to prove converse results of the type that if the resolvent has a rational root then $G_f$ is conjugate to a subgroup of $G$. What we require here is stronger in a way, for example we want that if $G_f = S_n$ and the resolvent is separable then it is irreducible over the rationals. Separable resolvents are two-way bridges between Galois theory and diophantine equations. We hope that one day they will lead to Galois-theoretic Lefschetz principles, so that Cohen's $p$-adic machinery \cite{Coh1972} can be brought to bear on the characteristic 0 setting. This article and the earlier works \cite{CD2017, Die2012} develop a framework for navigating some of these issues.

\bigskip

If $G_f = G$ then, by \cite[Lemma 5]{Die2012}, the polynomial
\begin{equation} \label{PhiDef}
\Phi(y) = \Phi(y; a_1, \ldots, a_n) = \prod_{\sig \in S_n / G}
\left(y - \sum_{\tau \in G} \prod_{i \le n} \alp_{\sig \tau(i)}^i \right)
\in \bZ[y;a_1,\ldots,a_n]
\end{equation}
has an integer root $y$. By \cite[Lemma 1]{Die2012}, if $a_1,\ldots,a_n \ll H$ then
\begin{equation} \label{ybound}
y = \sum_{\tau \in G} \prod_{i \le n} \alp_{\sig \tau(i)}^i \ll_n H^{O_n(1)}
\end{equation}
for some $\sig$. We will see that the total degree of $\Phi$ is $O_n(1)$. 

The more general family of resolvents
\[
\Phi_{\bw, \be, \fg}(y; a_1,\ldots,a_n) 
= \prod_{\sig \in S_n/G} (y - r_{\bw, \be, \fg}(\sig)) \in \bZ[y; a_1,\ldots,a_n],
\]
where 
\[
r_{\bw, \be, \fg}(\sig)
= \sum_{k \le |G|} w_k \sum_{\tau \in G} 
\prod_{i \le n} (\alp_{\sig \tau(i)} + \fg)^{ke_i},
\]
$\bw = (w_1,\ldots,w_{|G|}) \in \bN^{|G|}$, $\be = (e_1,\ldots,e_n) \in \bN^n$, and $\fg \in \bZ$, was introduced in \cite{CD2017}. Here we establish some additional features, in particular with regards to uniformity.

\begin{lemma} \label{ResolventProperty} If $G_f \le G$, then $\Psi$ has an integer root $y \ll_n \| \bw \|_\infty (H + |\fg|)^{O_n(\| \be \|_\infty)}$. Finally, the total degree of $\Psi$ is $O_n(1)$.
\end{lemma}

\begin{proof} For
\[
y = \sum_{k \le |G|} w_k \sum_{\tau \in G} \prod_{i \le n} (\alp_{\tau(i)} + \fg)^{k e_i},
\]
observe that if $\sig \in G_f \le G$ then $\sig(y) = y$, and so $y \in \bQ$ by the Galois correspondence. As $\Psi$ is monic, we then have $y \in \bZ$. The bound follows from \cite[Lemma 1]{Die2012}. 

Put $d= [S_n:G]$. Writing
\[
\Phi(y) = y^d + h_1(a_1,\ldots,a_n) y^{d-1} + \cdots + h_d(a_1,\ldots,a_n),
\]
it remains to prove that $h_1,\ldots,h_d$ are bounded-degree polynomials. By Vieta's formulas, these are elementary symmetric polynomials in the roots $r_{\bw,\be,\fg}(\sig)$
of $\Psi$, up to sign. Given $i \in \{1,2,\ldots,n\}$, the polynomial $h_i$ has bounded degree in $\alp_1,\ldots,\alp_n$. From the proof of the fundamental theorem of symmetric polynomials \cite[Chapter 7, \S 1, Theorem 3]{CLO2015}, any symmetric polynomial of degree $O_n(1)$ in $n$ variables also has degree $O_n(1)$ in the elementary symmetric polynomials. Therefore $\deg(h_i) \ll_n 1$. 
\end{proof}

The following is a refinement of \cite[Lemma 3]{CD2017}. 

\begin{lemma} \label{separable} Let $C, P \in \bN$ be large in terms of $n$, let $\fg = C^3 P$, and let $a_1,\ldots,a_n \in [-P,P]$ be such that the roots $\alp_1,\ldots,\alp_n$ of the polynomial $f(X) = X^n + a_1 X^{n-1} + \cdots + a_n$ are pairwise distinct. Then there exist $\bw, \be$, with $\| \bw \|_\infty, \| \be \|_\infty \le C$, such that $\Phi_{\bw, \be, \fg}(y; a_1,\ldots,a_n)$ is separable. 
\end{lemma}

\begin{proof} We know from \cite[Lemma 1]{Die2012} that if $|a_1|,\ldots,|a_n| \le P$ then 
\[
|\alp_1|,\ldots,|\alp_n| \le CP.
\]
Let us introduce the notation
\[
\ome_{\sig, \tau, \be, \fg} = \prod_{i \le n} (\alp_{\sig\tau(i)} + \fg)^{e_i}.
\]
Our immediate goal is to choose positive integers $e_1, \ldots, e_n \le C$ such that
\[
\ome_{\sig_1, \tau_1, \be, \fg} \ne \ome_{\sig_2, \tau_2, \be, \fg} 
\]
whenever $\sig_1 \ne \sig_2$.

Let $\sig_1, \sig_2 \in S_n/G$ with $\sig_1 \ne \sig_2$, and let $\tau_1,\tau_2 \in G$. Then $\sig_1 \tau_1 \ne \sig_2 \tau_2$, so for some $i \in \{1,2,\ldots,n\}$ we have $\alp_s \ne \alp_t$, where $s = \sig_1\tau_1(i)$ and $t = \sig_2\tau_2(i)$. Moreover, if $\ome_{\sig_1, \tau_1, \be, \fg} = \ome_{\sig_2, \tau_2, \be, \fg}$ then
\[
\left( \frac{\alp_s + \fg}{\alp_t + \fg} \right)^{e_i} = c,
\]
for some $c = c(\balp, \fg, e_1,\ldots,e_{i-1}, e_{i+1}, \ldots, e_n)$. We claim that this has at most one positive integer solution $e_i \le C$. The argument of $(\alp_s + \fg)/(\alp_t+\fg)$ is $O(C^{-2})$, and so its first $C$ powers are pairwise distinct. This establishes the claim. As $C$ is large and there are $O_n(C^{n-1})$ vectors to avoid out of $C^n$ possibilities, we conclude that there exist $e_1,\ldots,e_n \le C$ with the desired property.

Finally, we may choose positive integers $w_1, \ldots, w_{|G|} \le C$ such that the roots 
\[
r_{\bw, \be, \fg}(\sig) = \sum_{k \le |G|} w_k \sum_{\tau \in G} \ome_{\sig, \tau, \be, \fg}^k
\]
of $\Phi_{\bw, \be, \fg}$ are pairwise distinct. This is achieved by avoiding $O_n(1)$ proper, linear subspaces, as in the proof of \cite[Lemma 3]{CD2017}.
\end{proof}

The following is a refinement of \cite[Lemma 4]{CD2017}.

\begin{lemma} \label{GeneralResolvent}
Let
\[
g(T_1,\ldots,T_s,X)
\in \bZ[T_1,\ldots,T_s,X]
\]
be separable and monic of degree $n \ge 1$ in the variable $X$, let $G$ be its Galois group over $\bQ(T_1,\ldots,T_s)$, and let $K \le G$. Let $D$ be the total degree of $g$. Then there exists
\[
\Phi_{g,K}(T_1,\ldots,T_s,Y) \in \bZ[T_1,\ldots,T_s,Y]
\]
of total degree $O_{s,D}(1)$,
monic of degree $[S_n:K]$ in $Y$, such that:
\begin{enumerate}[(i)]
\item Each irreducible divisor has degree at least $[G:K]$ in $Y$.
\item If $t_1,\ldots,t_s \in \bZ$ and $g(t_1,\ldots,t_s,X)$ has Galois group $K$ over $\bQ$ then
$\Phi_{g,K}(t_1,\ldots,t_s,Y)$ has an integer root $y \ll_D ( |g| \cdot \| \bt \|_{\infty})^{O_{D}(1)} $.
\end{enumerate}
\end{lemma}

\begin{remark} As $n \in \{1,2,\ldots,D \}$, any dependence of implied constants on $n$ is controlled by the dependence on $D$.
\end{remark}

\begin{proof}
The discrimant $\Del$ of $g$ in $Y$ is a non-zero polynomial of degree $O_D(1)$ in $T_1,\ldots,T_s$, so there exist integers $t_1,\ldots,t_s \ll_D 1$ such that $\Del(t_1,\ldots,t_s) \ne 0$. We can then use Lemma \ref{separable} with $P \ll_D |g|$, in lieu of \cite[Lemma 3]{CD2017}, in the proof of \cite[Lemma 4]{CD2017}. This constructs $\Phi_{g,K}$.

From the proof of the fundamental theorem of symmetric polynomials \cite[Chapter 7, \S 1, Theorem 3]{CLO2015}, any symmetric polynomial of degree $O_s(1)$ in $s$ variables also has degree $O_s(1)$ in the elementary symmetric polynomials. We will use this fact to show that the total degree of $\Phi_{g,K}$ is $O_{s,D}(1)$. 

By construction, the total degree of $\Phi_{g,K}$ in $\alp_1(T_1,\ldots,T_s),\ldots,\alp_n(T_1,\ldots,T_s)$ from \cite[Lemma 4]{CD2017} is $O_{s,n}(1)$. The total degree of $\Phi_{g,K}$ in $g_1(T_1,\ldots,T_s), \ldots, g_n(T_1,\ldots,T_s)$ from \cite[Lemma 4]{CD2017} is therefore also $O_{s,n}(1)$, by the fact explained in the previous paragraph. Thus, the total degree in $T_1,\ldots,T_s$ is $O_{s,D}(1)$.

A root $y$ of $\Phi_{g,K}$ has the form
\[
y = \sum_{k \le |K|} w_k \sum_{\tau \in G} \prod_{i \le n} (\alp_{\tau(i)} + \fg)^{k e_i},
\]
for some positive integers $w_1,\ldots,w_{|K|},e_1,\ldots,e_n \ll_D 1$ and some integer $\fg \ll_D |g|$. By \cite[Lemma 1]{Die2012} we have 
\[
\alp_j \ll_n |g| \cdot \| \bt\|_{\infty}^{D} 
\qquad (1 \le j \le n),
\]
so $y \ll_D (|g| \cdot \| \bt \|_{\infty}^{D})^{O_n(1)}$.
\end{proof}

The following lemma serves as a strong converse to Lemma \ref{ResolventProperty}.

\begin{lemma} \label{converse} If $G_f = S_n$ and $\Psi = \Phi_{\bw, \be, \fg}$ is separable then $\Psi$ is irreducible over $\bQ$.
\end{lemma}

\begin{proof} Suppose for a contradiction that $\Psi(y) = \Psi_1(y) \Psi_2(y)$, where $\Psi_1(y), \Psi_2(y) \in \bQ[y]$ are non-constant polynomials. Let $r_{\sig_1} = r_{\bw, \be, \fg}(\sig_1)$ and $r_{\sig_2} = r_{\bw, \be, \fg}(\sig_2)$ be roots of $\Psi_1$ and $\Psi_2$ respectively. As $\Psi_1$ has rational coefficients and $\sig_2 \sig_1^{-1} \in G_f$, the complex number
\[
r_{\sig_2} = \sig_2 \sig_1^{-1} (r_{\sig_1})
\]
must be a root of $\Psi_1$, contradicting the separability of $\Psi$. 
\end{proof}

\begin{lemma} \label{hdef}
Let $a_1,\ldots,a_{n-3} \in \bZ \cap [-H,H]$, and let $L(a,b) \in \bQ[a,b]$ be linear. Assume that
\[
\Gal(X^n + a_1 X^{n-1} + \cdots + a_{n-3} X^3 + aX^2 + bX + L(a,b), \bQ(a,b)) = S_n.
\]
Then, associated to any proper subgroup of $S_n$, there exist $\bw, \be, \fg$, with $\| \bw \|_{\infty}, \| \be \|_\infty \ll_n 1$ and $\fg \ll H^{O(1)}(1 + |L|)$, such that 
\[
h(a,b,y) = \Phi_{\bw, \be, \fg} (y; a_1, \ldots, a_{n-3}, a, b, L(a,b))
\]
is irreducible over $\bQ$. The same is true with
\[
\Phi_{\bw, \be, \fg} (y; a_1, \ldots, a_{n-3}, a, L(a,b), b)
\]
or
\[
\Phi_{\bw, \be, \fg} (y; a_1, \ldots, a_{n-3}, L(a,b), a, b)
\]
in place of 
$\Phi_{\bw, \be, \fg} (y; a_1, \ldots, a_{n-3}, a, b, L(a,b))$, with the corresponding change to the hypothesis.
\end{lemma}

\begin{proof} Let $C \in \bN$ be a large constant, and let $P = H^C$. Write
\[
L(a,b) = \lam a + \mu b + \nu,
\]
and define
\[
G_0(X) = X^n + a_1 X^{n-1} + \cdots + a_{n-3}X^3 + \nu,
\quad G_1(X) = X^2 + \lam,
\quad G_2(X) = X + \mu.
\]
By \cite[Theorem 2.1]{Coh1981}, there can be at most $H^{O(1)} P^{3/2} \log P$ pairs $(a,b) \in [-P,P]^2$ of integers for which  $G_{G_0+aG_1+bG_2} \ne S_n$. Hence, there exist integers $a_0, b_0 \in [-P,P]$ such that $G_{G_0+a_0 G_1+b_0 G_2}  = S_n$. Now Lemma \ref{converse} tells us that $h(a_0,b_0,y)$ is irreducible over $\bQ$, when $\bw, \be, \fg$ are obtained by applying Lemma \ref{separable} to the polynomial 
\[
X^n + a_1 X^{n-1} + \cdots + a_{n-3}X^3 + a_0 X^2 + b_0X +L(a_0,b_0).
\]

Suppose for a contradiction that $h(a,b,y)$ is reducible over $\bQ$. Then, for some positive integer $k \le d-1$ and some $b_1(a,b),\ldots,b_k(a,b), c_1(a,b),\ldots,c_{d-k}(a,b) \in \bQ(a,b)$, we have
\[
h(a,b,y) = (y^k + b_1(a,b) y^{k-1} + \cdots + b_k(a,b))
(y^{k-d} + c_1(a,b) y^{k-d-1} + \cdots + c_{d-k}(a,b)).
\]
Specialising $a = a_0$ and $b = b_0$ contradicts the irreducibility of $h(a_0,b_0,y)$.

The final sentence of the lemma is confirmed by imitating the proof of its first assertion.
\end{proof}

\section{Sextic Galois theory}
\label{SexticSection}

In this section, we establish Theorem \ref{SexticThm}.
Let
\[
f(X) = X^6 - a_1 X^5 + a_2 X^4 - a_3 X^3 + a_4 X^2 - a_5 X + a_6 \in \bZ[X]
\]
be irreducible, and suppose its Galois group $G_f$ is not conjugate to $S_6$ or $A_6$. Then $G_f$ is conjugate to a subgroup of  $G_{72}$, $G_{48}$, or $H_{120}$, in the notation of Hagedorn \cite{Hagedorn}. Note that $G_{72}$ and $G_{48}$ are solvable, whereas $H_{120}$ is insolvable.

\subsection{A decic resolvent}
\label{DecicResolvent}

If $G_f$ is conjugate to a subgroup of $G_{72}$, then by \cite[Theorem 2]{Hagedorn} the polynomial $f_{10}(X)$ from that article has a rational root. As $f_{10}(X)$ is a monic, decic polynomial with integer coefficients, it follows that $f_{10}(y) = 0$ for some $y \in \bZ$.

\subsubsection{Non-degenerate quadruples for the decic resolvent} 
\label{NonDegenerateDecic}

Let $a_1,\ldots,a_4 \in [-H,H]$ be integers such that the vanishing of
\[
g(a_5, a_6, y) = f_{10}(y; a_1,\ldots,a_6)
\]
cuts out an absolutely irreducible surface with no rational lines. Expressions for the coefficients of $f_{10}$ as polynomials in $a_1,\ldots,a_6$ are given in the appendix of \cite{Hagedorn}, and we also computed them using \emph{Mathematica}. Note that
\[
f_{10}(y) = y^{10} - b_1 y^9 + \cdots - b_9 y + b_{10},
\]
where $b_i \ll H^i$ for all $i$ assuming $a_1,\ldots,a_6 \ll H$, and therefore all solutions have $y \ll H$. Note that a large part of the expression for $b_{10}$ is missing from \cite{Hagedorn}.

We adopt the strategy of \S \ref{QuinticNonDegenerate}. First we cover all but $O(H^{\eps + 2/\sqrt{10}})$ points by $O(H^{\eps+1/\sqrt{10}})$ curves. Next, we eliminate $y$, giving $F(d,e) = 0$ with $\deg(F) \ll 1$, and let $\cF(d,e)$ be an irreducible factor of $F(d,e)$. There are $O(H^{\eps+1/2})$ zeros $(d,e)$ unless $\cF$ is linear, by Theorem \ref{LopsidedBP}. In the latter case we obtain
\[
\cP(d,y) = \cF(d,e) = 0,
\]
say, where $\cP$ is irreducible over $\bQ$. Now $\cP$ cannot be linear, so by  Theorem \ref{LopsidedBP} we have
\[
\# \{ (d,y) \in (\bZ \cap [-H,H])^2: \cP(d,y) = 0 \} \ll H^{\eps+ 1/2}.
\]
The contribution from non-degenerate quadruples is $O(H^{\eps+9/2+1/\sqrt{10}})$.

\subsubsection{Degenerate quadruples for the decic resolvent} 
\label{DegenerateDecic}

We begin by discussing absolute irreducibility. Note that $g$ is decic in $a_5,a_6,y$, no matter the specialisation $a_1,\ldots,a_4 \in \bZ$. Thus, by \cite[Chapter V, Theorem 2A]{Schmidt}, there exist integer polynomials 
\[
p_1(a_1,\ldots,a_4), \ldots,p_s(a_1,\ldots,a_4)
\]
of degree $O(1)$ such that if $a_1,\ldots,a_4 \in \bC$ then
\[
g \text{ is not absolutely irreducible } \Leftrightarrow p_i(a_1,\ldots,a_4) = 0 \: (1 \le i \le s).
\]
The next lemma shows that $p_i$ is non-zero for some $i$.

\begin{lemma} If $a_1 = \cdots = a_4 = 0$ then $g(a_5,a_6,y)$ is absolutely irreducible.
\end{lemma}

\begin{proof} Let $a_1 = \cdots = a_4 = 0$. Then 
\begin{align*}
b_1 &= \cdots = b_4 = 0, \qquad b_5 = -123 a_5^2,
\qquad b_6 = 129a_6^2,
\qquad b_7 = 0, \\
b_8 &= 66 a_5^2 a_6, \qquad b_9 = 64 a_6^3, \qquad b_{10} = a_5^4,
\end{align*}
so 
\[
g(a_5,a_6,y) = y^{10} + 123 a_5^2 y^5 + 129 a_6^2 y^4 
+ 66 a_5^2 a_6 y^2 - 64a_6^3 y +a_5^4.
\]
This is monic in $y$, so it suffices to prove that
\[
g(a_5,1,y)
= y^{10} + 123 a_5^2 y^5 + 129  y^4 
+ 66 a_5^2  y^2 - 64 y +a_5^4
\]
is absolutely irreducible. \emph{Mathematica} assures us that $g(a_5,1,y)$ is irreducible over the rationals. Its Newton polygon is the convex hull of
\[
(0,10), (2,5), (0,4), (2,2), (0,1), (4,0),
\]
which has vertices $(0,1),(0,10),(4,0)$. As $\gcd(0,1,0,10,4,0)=1$, absolute irreducibility is ensured by \cite[Proposition $3$]{BCG2010}.
\end{proof}

Consequently, the quadruples $(a_1,\ldots,a_4)$ such that $g(a_5,a_6,y)$ is reducible over $\overline{\bQ}$ are all zeros of a fixed non-zero polynomial of degree $O(1)$. In particular, there are $O(H^3)$ such quadruples $(a_1,\ldots,a_4)$.

\bigskip 

Next, we discuss lines on the surface. The idea is to show that, for generic $a_1,\ldots,a_4$, the surface $\{ g(a_5,a_6,y) = 0\}$ contains no rational lines or, \emph{a fortiori}, no complex lines. A line on the surface has one of the following parametrisations:
\begin{enumerate}[I.]
\item $\{ (0,\bet,\gam) + t(1,b,c): t \in \bC \}$, for some $\bet,\gam,b,c \in \bC$
\item $\{ (\alp,0,\gam) + t(0,1,c): t \in \bC \}$, for some $\alp,\gam,c \in \bC$
\item $\{ (\alp,\bet,0) + t(0,0,1): t \in \bC \}$, for some $\alp, \bet \in \bC$.
\end{enumerate}

We use the `SymmetricReduction' command in \emph{Mathematica} to evaluate and store the polynomial $g(a_5,a_6,y)$ explicitly:

\bigskip

{\tiny
\begin{verbatim}
ClearAll["Global`*"];
r1 = (x1 + x2 + x3)*(x4 + x5 + x6); r2 = (x1 + x2 + x4)*(x3 + x5 + x6); r3 = (x1 + x2 + x5)*(x3 + x4 + x6);
r4 = (x1 + x2 + x6)*(x3 + x4 + x5); r5 = (x1 + x3 + x4)*(x2 + x5 + x6); r6 = (x1 + x3 + x5)*(x2 + x4 + x6);
r7 = (x1 + x3 + x6)*(x2 + x4 + x5); r8 = (x1 + x4 + x5)*(x2 + x3 + x6); r9 = (x1 + x4 + x6)*(x2 + x3 + x5);
r10 = (x1 + x5 + x6)*(x2 + x3 + x4); 
b1 = SymmetricPolynomial[1, {r1, r2, r3, r4, r5, r6, r7, r8, r9, r10}];
b2 = SymmetricPolynomial[2, {r1, r2, r3, r4, r5, r6, r7, r8, r9, r10}];
b3 = SymmetricPolynomial[3, {r1, r2, r3, r4, r5, r6, r7, r8, r9, r10}];
b4 = SymmetricPolynomial[4, {r1, r2, r3, r4, r5, r6, r7, r8, r9, r10}];
b5 = SymmetricPolynomial[5, {r1, r2, r3, r4, r5, r6, r7, r8, r9, r10}];
b6 = SymmetricPolynomial[6, {r1, r2, r3, r4, r5, r6, r7, r8, r9, r10}];
b7 = SymmetricPolynomial[7, {r1, r2, r3, r4, r5, r6, r7, r8, r9, r10}];
b8 = SymmetricPolynomial[8, {r1, r2, r3, r4, r5, r6, r7, r8, r9, r10}];
b9 = SymmetricPolynomial[9, {r1, r2, r3, r4, r5, r6, r7, r8, r9, r10}];
b10 = SymmetricPolynomial[10, {r1, r2, r3, r4, r5, r6, r7, r8, r9, r10}];
B1 = SymmetricReduction[b1, {x1, x2, x3, x4, x5, x6}, {a1, a2, a3, a4, a5, a6}][[1]];
B2 = SymmetricReduction[b2, {x1, x2, x3, x4, x5, x6}, {a1, a2, a3, a4, a5, a6}][[1]];
B3 = SymmetricReduction[b3, {x1, x2, x3, x4, x5, x6}, {a1, a2, a3, a4, a5, a6}][[1]];
B4 = SymmetricReduction[b4, {x1, x2, x3, x4, x5, x6}, {a1, a2, a3, a4, a5, a6}][[1]];
B5 = SymmetricReduction[b5, {x1, x2, x3, x4, x5, x6}, {a1, a2, a3, a4, a5, a6}][[1]];
B6 = SymmetricReduction[b6, {x1, x2, x3, x4, x5, x6}, {a1, a2, a3, a4, a5, a6}][[1]];
B7 = SymmetricReduction[b7, {x1, x2, x3, x4, x5, x6}, {a1, a2, a3, a4, a5, a6}][[1]];
B8 = SymmetricReduction[b8, {x1, x2, x3, x4, x5, x6}, {a1, a2, a3, a4, a5, a6}][[1]];
B9 = SymmetricReduction[b9, {x1, x2, x3, x4, x5, x6}, {a1, a2, a3, a4, a5, a6}][[1]];
B10 = SymmetricReduction[b10, {x1, x2, x3, x4, x5, x6}, {a1, a2, a3, a4, a5, a6}][[1]];
g = y^(10) - B1*y^9 + B2*y^8 - B3*y^7 + B4*y^6 - B5*y^5 + B6*y^4 - B7*y^3 + B8*y^2 - B9*y + B10;
\end{verbatim}
} \bigskip

Then we expand it as a polynomial in $t$, in each of the three cases. In Case I, the polynomial is $c^{10} t^{10}$ plus lower-order terms, and so $c = 0$. When we specialise $c = 0$, we obtain a monic, quartic polynomial $t$. Such a polynomial cannot vanish identically, so there are no lines of Type I. 

Similarly in Case II we must have $c=0$, and then we obtain a polynomial of degree at most 3 in $t$. Setting the cubic coefficient to 0 expresses
\[
\gam = a_1^2/4.
\]
We then substitute this into the equation obtained by setting the quadratic coefficient to 0, determining an equation $P(a_1,\ldots,a_4) = 0$, where $P \in \bZ[a_1,\ldots,a_4] \setminus \{ 0 \}$ has degree $O(1)$, because $\alp$ so happens to be eliminated as a result of this substitution. Thus, there are $O(H^3)$ integer quadruples $(a_1,\ldots,a_4)$ such that $\{ g(a_5,a_6,y) = 0\}$ contains a line of Type II.

In Case III the polynomial is monic of degree 10, so there are no lines of Type III. Having considered all cases, we conclude that there are  $O(H^3)$ integer quadruples $(a_1,\ldots,a_4)$ such that $\{ g(a_5,a_6,y) = 0\}$ contains a line.

\bigskip Finally, we estimate the contribution from degenerate quadruples. We showed that there are $O(H^3)$ degenerate quadruples $(a_1,\ldots,a_4) \in [-H,H]$. Let us now fix such a quadruple. By \cite[Lemma 2]{Die2012}, there are at most $O(1)$ values of $a_5 \in \bZ$ for which 
\[
X^6 - a_1 X^5 + \cdots - a_5 X + a_6 \in \bQ(a_6)[X]
\]
has non-$S_6$ Galois group over $\bQ(a_6)$. The contribution from these choices of $a_5$ is $O(H^4)$. Now let $a_5 \in [-H,H]$ be such that the Galois group is $S_6$.

Let $C \in \bN$ be a large constant, and let $P = H^C$. By \cite[Theorem 2.1]{Coh1981}, there exists $a_6^* \in \bZ \cap [-P,P]$ such that
\[
\Gal(X^6 - a_1 X^5 + \cdots + a_4X^2 - a_5 X + a_6^*, \bQ) = S_6.
\]
By Lemma \ref{converse}, the polynomial
\[ \Phi_{\bw,\be,\fg}(y;a_1,\ldots,a_5, a_6^*) \in \bQ[y]
\]
is irreducible, where $\bw,\be,\fg$ are obtained by applying Lemma \ref{separable} to the polynomial 
\[
X^6 - a_1 X^5 + \cdots + a_4X^2 - a_5 X + a_6^*.
\]

It then follows that
\[
h(a_6,y) := \Phi_{\bw,\be,\fg}(y;a_1,\ldots,a_5, a_6) \in \bQ[a_6,y]
\]
is irreducible. Indeed, it is monic in $y$, and if we were to have
\[
h(a_6,y) = h_1(a_6,y) h_2(a_6,y)
\]
for some $h_1,h_2$ of positive degrees in $y$, then
\[
h(a_6^*,y) = h_1(a_6^*,y) h_2(a_6^*,y)
\]
would be a non-trivial factorisation of $h(a_6^*,y)$.

Further, by Lemma \ref{ResolventProperty}, if $a_6 \in \bZ \cap [-H,H]$ and $G_f \le G_{72}$ then $h(a_6,y)$ has an integer root $y \ll H^{O(1)}$. By Theorem \ref{LopsidedBP}, the diophantine equation
$h(a_6,y) = 0$ has $O(H^{\eps+|G_{72}|/720}) = O(H^{\eps+0.1})$ integer solutions $a_6,y$ with $a_6 \ll H$ and $y \ll H^{O(1)}$. The contribution from degenerate quadruples, when $G_f \le G_{72}$, is therefore $O(H^{4.1+\eps})$.

We conclude that
\[
N_G \ll H^{9/2+\eps+1/\sqrt{10}} \qquad (G \le G_{72}).
\]

\subsection{A resolvent of degree 15}

\subsubsection{Non-degenerate quadruples for the resolvent of degree 15}

For $G_f \le G_{48}$, we use the resolvent
\[
f_{15}(y) = y^{15} - c_1 y^{14} + \cdots + c_{14}y - c_{15}
\]
from \cite{Hagedorn},
where $c_i \ll H^i$ for all $i$ assuming $a_1,\ldots,a_6 \ll H$. Expressions for $c_1,\ldots,c_{15}$ as polynomials in $a_1,\ldots,a_6$ are given in the appendix of \cite{Hagedorn}, and we also computed them using \emph{Mathematica}. Note that the former contains a few small errors. The argument of \S \ref{NonDegenerateDecic} bounds the contribution from this case by a constant times
\[
H^{4+0.5+\eps+1/\sqrt{15}} \ll H^{4.76}.
\]
We provide the code below.

\bigskip

{\tiny
\begin{verbatim}
ClearAll["Global`*"];
r1 = (x1*x2) + (x3*x4) + (x5*x6); r2 = (x1*x2) + (x3*x5) + (x4*x6); r3 = (x1*x2) + (x3*x6) + (x4*x5);
r4 = (x1*x3) + (x2*x4) + (x5*x6); r5 = (x1*x3) + (x2*x5) + (x4*x6); r6 = (x1*x3) + (x2*x6) + (x4*x5);
r7 = (x1*x4) + (x2*x3) + (x5*x6); r8 = (x1*x4) + (x2*x5) + (x3*x6); r9 = (x1*x4) + (x2*x6) + (x3*x5);
r10 = (x1*x5) + (x2*x3) + (x4*x6); r11 = (x1*x5) + (x2*x4) + (x3*x6); r12 = (x1*x5) + (x2*x6) + (x3*x4);
r13 = (x1*x6) + (x2*x3) + (x4*x5); r14 = (x1*x6) + (x2*x4) + (x3*x5); r15 = (x1*x6) + (x2*x5) + (x3*x4);
b1 = SymmetricPolynomial[1, {r1, r2, r3, r4, r5, r6, r7, r8, r9, r10, r11, r12, r13, r14, r15}];
b2 = SymmetricPolynomial[2, {r1, r2, r3, r4, r5, r6, r7, r8, r9, r10, r11, r12, r13, r14, r15}];
b3 = SymmetricPolynomial[3, {r1, r2, r3, r4, r5, r6, r7, r8, r9, r10, r11, r12, r13, r14, r15}];
b4 = SymmetricPolynomial[4, {r1, r2, r3, r4, r5, r6, r7, r8, r9, r10, r11, r12, r13, r14, r15}];
b5 = SymmetricPolynomial[5, {r1, r2, r3, r4, r5, r6, r7, r8, r9, r10, r11, r12, r13, r14, r15}];
b6 = SymmetricPolynomial[6, {r1, r2, r3, r4, r5, r6, r7, r8, r9, r10, r11, r12, r13, r14, r15}];
b7 = SymmetricPolynomial[7, {r1, r2, r3, r4, r5, r6, r7, r8, r9, r10, r11, r12, r13, r14, r15}];
b8 = SymmetricPolynomial[8, {r1, r2, r3, r4, r5, r6, r7, r8, r9, r10, r11, r12, r13, r14, r15}];
b9 = SymmetricPolynomial[9, {r1, r2, r3, r4, r5, r6, r7, r8, r9, r10, r11, r12, r13, r14, r15}];
b10 = SymmetricPolynomial[10, {r1, r2, r3, r4, r5, r6, r7, r8, r9, r10, r11, r12, r13, r14, r15}];
b11 = SymmetricPolynomial[11, {r1, r2, r3, r4, r5, r6, r7, r8, r9, r10, r11, r12, r13, r14, r15}];
b12 = SymmetricPolynomial[12, {r1, r2, r3, r4, r5, r6, r7, r8, r9, r10, r11, r12, r13, r14, r15}];
b13 = SymmetricPolynomial[13, {r1, r2, r3, r4, r5, r6, r7, r8, r9, r10, r11, r12, r13, r14, r15}];
b14 = SymmetricPolynomial[14, {r1, r2, r3, r4, r5, r6, r7, r8, r9, r10, r11, r12, r13, r14, r15}];
b15 = SymmetricPolynomial[15, {r1, r2, r3, r4, r5, r6, r7, r8, r9, r10, r11, r12, r13, r14, r15}];
B1 = SymmetricReduction[b1, {x1, x2, x3, x4, x5, x6}, {a1, a2, a3, a4, a5, a6}][[1]]
B2 = SymmetricReduction[b2, {x1, x2, x3, x4, x5, x6}, {a1, a2, a3, a4, a5, a6}][[1]]
B3 = SymmetricReduction[b3, {x1, x2, x3, x4, x5, x6}, {a1, a2, a3, a4, a5, a6}][[1]]
B4 = SymmetricReduction[b4, {x1, x2, x3, x4, x5, x6}, {a1, a2, a3, a4, a5, a6}][[1]]
B5 = SymmetricReduction[b5, {x1, x2, x3, x4, x5, x6}, {a1, a2, a3, a4, a5, a6}][[1]]
B6 = SymmetricReduction[b6, {x1, x2, x3, x4, x5, x6}, {a1, a2, a3, a4, a5, a6}][[1]]
B7 = SymmetricReduction[b7, {x1, x2, x3, x4, x5, x6}, {a1, a2, a3, a4, a5, a6}][[1]]
B8 = SymmetricReduction[b8, {x1, x2, x3, x4, x5, x6}, {a1, a2, a3, a4, a5, a6}][[1]]
B9 = SymmetricReduction[b9, {x1, x2, x3, x4, x5, x6}, {a1, a2, a3, a4, a5, a6}][[1]]
B10 = SymmetricReduction[b10, {x1, x2, x3, x4, x5, x6}, {a1, a2, a3, a4, a5, a6}][[1]]
B11 = SymmetricReduction[b11, {x1, x2, x3, x4, x5, x6}, {a1, a2, a3, a4, a5, a6}][[1]]
B12 = SymmetricReduction[b12, {x1, x2, x3, x4, x5, x6}, {a1, a2, a3, a4, a5, a6}][[1]]
B13 = SymmetricReduction[b13, {x1, x2, x3, x4, x5, x6}, {a1, a2, a3, a4, a5, a6}][[1]]
B14 = SymmetricReduction[b14, {x1, x2, x3, x4, x5, x6}, {a1, a2, a3, a4, a5, a6}][[1]]
B15 = SymmetricReduction[b15, {x1, x2, x3, x4, x5, x6}, {a1, a2, a3, a4, a5, a6}][[1]]
\end{verbatim}   
} \bigskip

\subsubsection{Degenerate quadruples for the resolvent of degree 15}

We begin by discussing absolute irreducibility. Note that $g$ has degree $15$ in $a_5,a_6,y$, no matter the specialisation $a_1,\ldots,a_4 \in \bZ$. Thus, by \cite[Chapter V, Theorem 2A]{Schmidt}, there exist integer polynomials $p_1(a_1,\ldots,a_4), \ldots, p_s(a_1,\ldots,a_4)$ of degree $O(1)$ such that if $a_1,\ldots,a_4 \in \bC$ then
\[
g \text{ is not absolutely irreducible } \Leftrightarrow p_i(a_1,\ldots,a_4) = 0 \: (1 \le i \le s).
\]
The next lemma shows that $p_i$ is non-zero for some $i$.

\begin{lemma} The polynomial
\[
g(a_5,a_6,y) := f_{15}(y; 0,0,0,0,a_5,a_6)
\]
is absolutely irreducible.
\end{lemma}

\begin{proof}
As $g(a_5,a_6,y)$ is monic in $y$, it suffices to prove that $g(a_5,1,y)$ is absolutely irreducible. We compute using \emph{Mathematica} that
\begin{align*}
g(a_5,1,y) = 
32 a_5^6 + 1296 a_5^2 y + 792 a_5^4 y^2 - 1728 y^3 - 96 a_5^2 y^4 - 
 353 a_5^4 y^5 \\
 \qquad - 1232 y^6 + 288 a_5^2 y^7 + 453 y^9 - 21 a_5^2 y^{10} - 
 42 y^{12} + y^{15}
\end{align*}
is irreducible over the rationals. Its Newton polygon is the convex hull of
\[
(6,0), (2,1), (4,2), (0,3), (2,4), (4,5), (0, 6), (2,7), (0,9), (2,10), (0,12), (0, 15),
\]
which has vertices $(6,0), (0,3), (0,15), (2, 1)$. As $\gcd(6,0,0,3,0,15,2,1)=1$, absolute irreducibility is assured by \cite[Proposition $3$]{BCG2010}.
\end{proof}

Consequently, the quadruples $(a_1,\ldots,a_4)$ such that $g(a_5,a_6,y)$ is reducible over $\bC$ are all zeros of a fixed non-zero polynomial of degree $O(1)$. In particular, there are $O(H^3)$ such quadruples $(a_1,\ldots,a_4)$.

\bigskip

Next, we discuss lines on the surface. We have the same trichotomy as in the previous subsection. In Case I, the polynomial is $c^{15} t^{15}$ plus lower-order terms, so $c=0$. When we specialise $c=0$, we obtain $32t^6$ plus lower-order terms. Such a polynomial cannot vanish identically, so there are no lines of Type I.

\bigskip

Similarly in Case II we must have $c=0$, and then we obtain a polynomial of degree at most 4 in $t$. We set the quartic, cubic and quadratic coefficients to 0, computing them using \emph{Mathematica}. Calling these $P_4,P_3,P_2$ respectively, we obtain 
\begin{equation} \label{432}
P_4(a_1,a_2,\gam) = P_3(a_1,\ldots,a_4,\alp,\gam) = P_2(a_1,\ldots,a_4,\alp,\gam) = 0,
\end{equation}
and $P_4,P_3,P_2$ are non-zero polynomials of degree $O(1)$. It so happens that 
\begin{align*}
P_4(a_1,a_2,\gam) &= -49 a_1^6 + 315 a_1^4 a_2 - 648 a_1^2 a_2^2 + 432 a_2^3 + 
 189 a_1^4 \gam - 1080 a_1^2 a_2 \gam \\
 & \qquad + 1296 a_2^2 \gam + 
 432 a_1^2 \gam^2 - 1728 \gam^3
\end{align*}
does not involve $a_3,a_4,\alp$. Now $R(a_1,\ldots,a_4,\gam) =0$, where $R$ is the resultant of $P_3$ and $P_2$ in the variable $\alp$. We explicitly evaluate $R$ using \emph{Mathematica}.

Let $P(a_1,\ldots,a_4)$ be the resultant of $P_4$ and $R$ in the variable $\gam$. We can compute 
$
P(1,0,0,0)
$
using \emph{Mathematica}, by specialising before taking the resultant. The outcome is non-zero, so $P$ is not the zero polynomial. Furthermore
\[
P(a_1,\ldots,a_4) = 0
\]
whenever $(a_1,a_2,a_3,a_4,\alp,\gam) \in \bC^6$ is a solution to \eqref{432}. Thus, there are $O(H^3)$ integer quadruples $(a_1,\ldots,a_4)$ such that $\{g(a_5,a_6,y)=0\}$ contains a line of Type II.

\bigskip

In Case III the polynomial is monic of degree 15, so there are no rational lines of Type III. Having considered all cases, we conclude that there are  $O(H^3)$ integer quadruples $(a_1,\ldots,a_4)$ such that $\{ g(a_5,a_6,y) = 0\}$ contains a complex line.

\bigskip

The upshot is that there are $O(H^3)$ degenerate quadruples to consider. Now the reasoning of the  previous subsection bounds their contribution by  $O(H^{\eps+61/15})$. We conclude that if $G \le G_{48}$ then
\[
N_G \ll H^{4+0.5+\eps+1/\sqrt{15}} \le H^{4.76}.
\]

This completes the proof of the second statement in Theorem \ref{SexticThm}.

\subsection{Insolvable sextics}

For subgroups of $H_{120}$, we use Stauduhar's \cite{Stauduhar} resolvent
\[
\Psi(y;a_1,\ldots,a_6) = \prod_{\sig \in S_6/H_{120}} (y - \sig \tet) \in \bZ[y; a_1,\ldots,a_6],
\]
where
\begin{align*}
\tet &= (\alp_1 \alp_2 + \alp_3 \alp_5 + \alp_4 \alp_6)
(\alp_1 \alp_3 + \alp_4 \alp_5 + \alp_2 \alp_6)
(\alp_3 \alp_4 + \alp_1 \alp_6 + \alp_2 \alp_5) \\
&\qquad \cdot
(\alp_1 \alp_5 + \alp_2 \alp_4 + \alp_3 \alp_6)
(\alp_1 \alp_4 + \alp_2 \alp_3 + \alp_5 \alp_6).
\end{align*}
Here, if 
\[
f(X) = X^6 - a_1 X^5 + a_2 X^4 - a_3 X^3 + a_4 X^2 - a_5 X + a_6 \in \bZ[X]
\]
is irreducible, we write $\alp_1,\ldots,\alp_6$ for the roots of $f$. At first it may appear that Stauduhar uses right cosets in \cite{Stauduhar}, but in modern language these $\sig H_{120}$ are more commonly referred to as left cosets; see the footnote on \cite[p. 983]{Stauduhar}.

As $H_{120}$ is generated by $(126)(354)$, $(12345)$ and $(2354)$, from \cite[Table 1]{Stauduhar} where it is denoted $G_{120}$, we see that $\tet$ is $H_{120}$-invariant. Thus, if $G_f \le \tau H_{120} \tau^{-1}$ for some $\tau \in S_6$, then $\tau \tet$ is a $G_f$-invariant algebraic integer and so $\tau \tet \in \bZ$. The upshot is that if $G_f$ is conjugate to a subgroup of $H_{120}$ then $\Psi$ has an integer root.

One can find the cycle types of the conjugacy classes of $H_{120}$, which is also known as 6T14, in \cite[Table 6C]{BM1983}, to see that $H_{120}$ has no 2-cycles and no 3-cycles. It follows that $\{ (1),(12),(13),(14),(15),(16)\}$ constitutes a complete set of left coset representatives for $H_{120}$ in $S_6$. This enables us to construct the resolvent in \emph{Mathematica}:

\bigskip

{\tiny
\begin{verbatim}
ClearAll["Global`*"];
r1 = (x1*x2 + x3*x5 + x4*x6)*(x1*x3 + x4*x5 + x2*x6)*(x3*x4 + x1*x6 + x2*x5)*(x1*x5 + x2*x4 + x3*x6)
*(x1*x4 + x2*x3 + x5*x6);
r2 = (x2*x1 + x3*x5 + x4*x6)*(x2*x3 + x4*x5 + x1*x6)*(x3*x4 + x2*x6 + x1*x5)*(x2*x5 + x1*x4 + x3*x6)
*(x2*x4 + x1*x3 + x5*x6);
r3 = (x3*x2 + x1*x5 + x4*x6)*(x3*x1 + x4*x5 + x2*x6)*(x1*x4 + x3*x6 + x2*x5)*(x3*x5 + x2*x4 + x1*x6)
*(x3*x4 + x2*x1 + x5*x6);
r4 = (x4*x2 + x3*x5 + x1*x6)*(x4*x3 + x1*x5 + x2*x6)*(x3*x1 + x4*x6 + x2*x5)*(x4*x5 + x2*x1 + x3*x6)
*(x4*x1 + x2*x3 + x5*x6);
r5 = (x5*x2 + x3*x1 + x4*x6)*(x5*x3 + x4*x1 + x2*x6)*(x3*x4 + x5*x6 + x2*x1)*(x5*x1 + x2*x4 + x3*x6)
*(x5*x4 + x2*x3 + x1*x6);
r6 = (x6*x2 + x3*x5 + x4*x1)*(x6*x3 + x4*x5 + x2*x1)*(x3*x4 + x6*x1 + x2*x5)*(x6*x5 + x2*x4 + x3*x1)
*(x6*x4 + x2*x3 + x5*x1);
b1 = SymmetricPolynomial[1, {r1, r2, r3, r4, r5, r6}];
b2 = SymmetricPolynomial[2, {r1, r2, r3, r4, r5, r6}];
b3 = SymmetricPolynomial[3, {r1, r2, r3, r4, r5, r6}];
b4 = SymmetricPolynomial[4, {r1, r2, r3, r4, r5, r6}];
b5 = SymmetricPolynomial[5, {r1, r2, r3, r4, r5, r6}];
b6 = SymmetricPolynomial[6, {r1, r2, r3, r4, r5, r6}];
B1 = SymmetricReduction[b1, {x1, x2, x3, x4, x5, x6}, {a1, a2, a3, a4, a5, a6}][[1]]; Export["B1.txt", B1];
B2 = SymmetricReduction[b2, {x1, x2, x3, x4, x5, x6}, {a1, a2, a3, a4, a5, a6}][[1]]; Export["B2.txt", B2];
B3 = SymmetricReduction[b3, {x1, x2, x3, x4, x5, x6}, {a1, a2, a3, a4, a5, a6}][[1]]; Export["B3.txt", B3];
B4 = SymmetricReduction[b4, {x1, x2, x3, x4, x5, x6}, {a1, a2, a3, a4, a5, a6}][[1]]; Export["B4.txt", B4];
B5 = SymmetricReduction[b5, {x1, x2, x3, x4, x5, x6}, {a1, a2, a3, a4, a5, a6}][[1]]; Export["B5.txt", B5];
B6 = SymmetricReduction[b6, {x1, x2, x3, x4, x5, x6}, {a1, a2, a3, a4, a5, a6}][[1]]; Export["B6.txt", B6];
\end{verbatim}
} \bigskip

\subsubsection{Non-degenerate quadruples}

In due course, we will define a non-zero polynomial
\[
R(a_1,\ldots,a_4) \in \bZ[a_1,\ldots,a_4]
\]
with coefficients and degree $O_n(1)$.
Let $a_1,\ldots,a_4 \in [-H,H]$ be integers such that the vanishing of
\[
g(a_5, a_6, y) = \Psi(y; a_1,\ldots,a_6)
\]
cuts out an absolutely irreducible surface with no rational lines, and such that
\begin{equation} \label{SexticWeird}
(5 a_1^2 - 18 a_2) R(a_1,\ldots,a_4) \ne 0.
\end{equation}
We wish to count integer zeros of $g$ with $a_5, a_6 \ll H$, and for these we have $y \ll H^M$ for some absolute constant $M \in \bN$, by \cite[Lemma 1]{Die2012}. Following \S \ref{QuinticNonDegenerate}, we can cover all but $O(H^{\eps + 2/\sqrt{6}})$ points by $O(H^{\eps+1/\sqrt{6}})$ curves using Theorem \ref{BrowningThm}. Then we eliminate $y$, giving $F(a_5,a_6) = 0$ with $\deg(F) \ll 1$, and let $\cF(a_5,a_6)$ be an irreducible factor of $F(a_5,a_6)$. There are $O(H^{\eps+1/2})$ zeros $(a_5,a_6)$ unless $\cF$ is linear, so we now assume the latter. We divide into two cases according to whether or not the coefficient of $a_6$ in $\cF$ vanishes.

\bigskip

\textbf{Case 1: $a_6 = \fH a_5 + \fL$.}
Substituting this into $g(a_5,a_6,y) = 0$ and factorising, we obtain
\[
\cP(a_5,y) = \cF(a_5,a_6) = 0,
\]
where $\cP$ is irreducible over $\bQ$. Moreover, since there are no rational lines, the polynomial $\cP$ must be non-linear. By Theorem \ref{LopsidedBP}, we have
\[
\# \{ (a_5,y) \in \bZ^2 \cap [-H,H] \times [-CH^M,CH^M]: \cP(a_5,y) = 0 \} \ll H^{\eps+ M/D},
\]
where $C$ is a large, positive constant and
\[
2 \le D := \max \{ \deg_{a_5}(\cP), M \deg_y(\cP) \} \ll 1.
\]

If $D \ge 2M$ then the bound is adequate, so we may suppose that $2 \le D \le 2M-1$, and in particular $\deg_y (\cP) \le 1$. As $\cP(a_5,y)$ divides
\[
g(a_5, \fH a_5 + \fL,y) = y^6 + \sum_{i=0}^5 P_i(a_5) y^i,
\]
we infer that $\cP(a_5,y)$ is a rational multiple of $y - F(a_5)$, where $F(a_5) \in \bQ[a_5]$ has degree $\deg_{a_5}(\cP)$. Thus, we may assume that
\[
\cP(a_5,y) = y - F(a_5).
\]
Next, write
\[
(y - F(a_5)) \left (y^5 + \sum_{i = 0}^4  f_i(a_5) y^i \right)
= y^6 + \sum_{i=0}^5 P_i(a_5) y^i.
\]

Equating coefficients in $y$ yields
\begin{align*}
f_4(a_5) - F(a_5) &= P_5(a_5) \\
f_3(a_5) - f_4(a_5) F(a_5) &= P_4(a_5) \\
f_2(a_5) - f_3(a_5) F(a_5) &= P_3(a_5) \\
f_1(a_5) - f_2(a_5) F(a_5) &= P_2(a_5) \\
f_0(a_5) - f_1(a_5) F(a_5) &= P_1(a_5) \\
-f_0(a_5) F(a_5) &= P_0(a_5).
\end{align*}
Using \emph{Mathematica}, we compute that 
\[
\deg(P_i) \le 2(6-i) \qquad (0 \le i \le 5).
\]
The argument of \S \ref{QuinticNonDegenerate} then delivers a contradiction, unless $\deg(F) = 2$. We now assume the latter.

From the equations above, we have
\[
f_4(a_5) = P_5(a_5) + F(a_5),
\]
so
\[
f_3(a_5) = P_4(a_5) + F(a_5) (P_5(a_5) + F(a_5)).
\]
Continuing make the substitutions, we obtain 
\[
\fP = 0
\]
for some explicit polynomial $\fP$ in $F(a_5),P_5(a_5),\ldots,P_0(a_5)$. Let
\[
F(a_5) = \fA a_5^2 + \fB a_5 + \fG,
\]
where $\fA,\fB,\fG \in \bC$, and substitute this into $\fP$ to give
\[
P = 0 \in \bQ[a_1,\ldots,a_4, \fA, \fB, \fG, \fH, \fL][a_5],
\]
for some polynomial $P$ that we computed using \emph{Mathematica}. With the coefficients of the resolvent having been evaluated and stored, the code is as follows:

\bigskip
{\tiny
\begin{verbatim}
F = A a5^2 + B a5 + G; a6 = H a5 + L;
f4 = F - B1; f3 = f4 F + B2; f2 = f3 F - B3; f1 = f2 F + B4; f0 = f1 F - B5;
P = f0 F + B6;
\end{verbatim}
} \bigskip

For $i=0,1,\ldots,12$, denote by $c_i$ the coefficient of $a_5^i$ in $P$. We have
\[
0 = c_{12} = (\fA - 32)(\fA-2)^5,
\]
so $\fA \in \{2,32\}$. If $\fA = 2$ then $c_{11} = -30 (a_1 - 6 \fH)^5$, whereupon $\fH = a_1/6$ and 
\[
0 = c_{10} = \frac{5 (5 a_1^2 - 18 a_2)^5}{1296},
\]
contradicting \eqref{SexticWeird}. 

Therefore $\fA = 32$. The equation $c_{11} = 0$ then enables us to write $\fB$ as an explicit polynomial in $\fH,a_1,\ldots,a_4$. We substitute both of these data into the other coefficients.
Finally, taking resultants yields 
\[
R(a_1,a_2,a_3,a_4)=0,
\]
where
\begin{align*}
R_1 &= \Res(c_{10},c_9,\fG), \quad
R_2 = \Res(c_{10}, c_8,\fG), \quad
R_3 = \Res(c_{10},c_7,\fG),
\quad R_4 = \Res(c_{10},c_6,\fG), \\
R_5 &= \Res(R_1,R_2,\fH),
\qquad R_6 = \Res(R_3,R_4,\fH),\qquad
R = \Res(R_5,R_6,\fL),
\end{align*}
again contradicting \eqref{SexticWeird}. \emph{Mathematica} assures us that $R(1,0,0,1) \ne 0$, so in particular $R$ is not the zero polynomial. 

\bigskip

\textbf{Case 2: $a_5 = \fH$ is constant.}
Substituting this into $g(a_5,a_6,y) = 0$ and factorising, we obtain
\[
\cP(a_6,y) = \cF(a_5,a_6) = 0,
\]
where $\cP$ is irreducible over $\bQ$. Following the argument of the previous case, we find that
\[
\deg(P_i) < 2(6-i) \qquad (0 \le i \le 5),
\]
and the reasoning of \S \ref{QuinticNonDegenerate} delivers a contradiction.

\bigskip

The total contribution from non-degenerate quadruples is $O(H^{\eps+ 9/2+ 1/\sqrt{6}})$.

\subsubsection{Degenerate quadruples}

We begin by discussing absolute irreducibility. One can check, using \emph{Mathematica}, that $g$ is dodecic in $a_5,a_6,y$, no matter the specialisation $a_1,\ldots,a_4 \in \bQ$. Thus, by \cite[Chapter V, Theorem 2A]{Schmidt}, there exist integer polynomials 
\[
p_1(a_1,\ldots,a_4), \ldots,p_s(a_1,\ldots,a_4)
\]
of degree $O(1)$ such that if $a_1,\ldots,a_4 \in \bC$ then
\[
g \text{ is not absolutely irreducible } \Leftrightarrow p_i(a_1,\ldots,a_4) = 0 \: (1 \le i \le s).
\]
The next lemma shows that $p_i$ is non-zero for some $i$.

\begin{lemma} If $a_1 = \cdots = a_4 = 0$ then $ \{ g(a_5,a_6,y) = 0 \}$ is absolutely irreducible.
\end{lemma}

\begin{proof} Using \emph{Mathematica}, we compute that
\begin{align*}
g(a_5,a_6,y) &= \Psi(y; 0,0,0,0,a_5,a_6) = y^6 - 42 a_5^2 y^5 + 360 a_5^4 y^4 - (1360 a_5^6 - 46656 a_6^5) y^3 \\ &\qquad \qquad \qquad \qquad \qquad \qquad
+ (2640 a_5^8 - 34992 a_5^2 a_6^5)y^2  - 2592 a_5^{10} y + 1024 a_5^{12}.
\end{align*}
As $g$ is monic in $y$, it suffices to prove that
\[
g(1,a_6,y)
= 
y^6 - 42 y^5 + 360  y^4 - (1360 - 46656 a_6^5) y^3 
+ (2640 - 34992 a_6^5)y^2  - 2592  y + 1024 
\]
is absolutely irreducible. \emph{Mathematica} tells us that $g(1,a_6,y)$ is irreducible over $\bQ$. Its Newton polygon is the convex hull of
\[
(0,6),(0,5),(0,4),(5,3),(5,2),(0,1),(0,0),
\]
which has vertices $(0,0),(5,2),(5,3),(0,6)$. As $\gcd(0,0,5,2,5,3,0,6)=1$, absolute irreducibility is secured by \cite[Proposition $3$]{BCG2010}.
\end{proof}

Consequently, the quadruples $(a_1,\ldots,a_4)$ such that $g(a_5,a_6,y)$ is reducible over $\bC$ are all zeros of a fixed non-zero polynomial of degree $O(1)$. In particular, there are $O(H^3)$ such quadruples $(a_1,\ldots,a_4)$.

\bigskip

Next, we discuss lines on the surface, using the trichotomy from the previous two subsections. If there is a line of Type I, then
\[
\Psi( \gam + ct; a_1, \ldots, a_4, t, \bet + bt) = 0,
\]
for some $\bet, \gam, b,c \in \bC$.
The left hand side is $1024 t^{12}$ plus lower order terms, so there are no lines of this type. If there is a line of Type III, then
\[
\Psi(t; a_1, \ldots, a_4, \alp, \bet) = 0,
\]
for some $\alp,\bet \in \bC$.
The left hand side is monic of degree 6, so there are no lines of this type. If there is a line of Type II, then
\[
\Psi( \gam + ct; a_1, \ldots, a_4, \alp, t) = 0
\]
for some $\alp,\gam,c \in \bC$. The coefficient of $t^8$ is
\[
2401 a_1^{12} - 30870 a_1^{10} a_2 + 162729 a_1^8 a_2^2 - 450576 a_1^6 a_2^3 + 
 692064 a_1^4 a_2^4 - 559872 a_1^2 a_2^5 + 186624 a_2^6,
\]
so there are $O(H^3)$ integer quadruples $(a_1,\ldots,a_4)$ such that $\{ g(a_5,a_6,y) = 0 \}$ contains a complex line.

Finally, there are $O(H^3)$ exceptions to \eqref{SexticWeird}. Thus, there are $O(H^3)$ degenerate quadruples, and the reasoning of the previous subsections bounds their contribution by  $O(H^{4+\eps+1/6})$. We conclude that if $G \le H_{120}$ then
\[
N_G \ll H^{9/2+\eps+1/\sqrt{6}}.
\]

This completes the proof of Theorem \ref{SexticThm}.

\section{Auxiliary results}
\label{AuxSection}

\subsection{An almost-uniform version of Hilbert's irreducibility theorem}

\begin{thm} \label{HIT}
Let $g(T,X_1,\ldots,X_s,Y) \in \bZ[T,X_1,\ldots,X_s,Y]$ be irreducible of total degree $D$, and monic of degree $d \ge 2$ in the variable $Y$. Then there are $O_{s,D,\eps}(|g|^{\eps} H^{\eps+1/2})$
integers $t \in [-H,H]$ such that \[
g(t,X_1,\ldots,X_s,Y) \in \bZ[X_1,\ldots,X_s,Y]
\]
is reducible.
\end{thm}

\begin{remark} The point is that we have exerted strong quantitative control over the dependence on the coefficients of $g$. In S. D. Cohen's work \cite{Coh1981} the dependence is $|g|^{O(1)}$, but in practice it is more useful to have $O(|g|^\eps)$ dependence. As $d \in \{2,3,\ldots,D\}$, any dependence on $d$ is controlled by the arbitrary dependence on $D$.
\end{remark}

\begin{proof}
We begin by reducing to the case $s=0$. By \cite[Theorem 2.1]{Coh1981}, there exist integers $x_1,\ldots,x_s \ll |g|^{O_{s,D}(1)}$ such that $g(T,x_1,\ldots,x_s,Y)$ is irreducible. Suppose $t \in \bZ$ and $g(t,X_1,\ldots,X_s,Y)$ is reducible. Then there exist 
\[
g_1(X_1,\ldots,X_s,Y), 
g_2(X_1,\ldots,X_s,Y) \in \bZ[X_1,\ldots,X_s,Y],
\]
monic of positive degrees in $Y$, such that
\[
g(t,X_1,\ldots,X_s,Y) =
g_1(X_1,\ldots,X_s,Y)
g_2(X_1,\ldots,X_s,Y).
\]
Now
\[
g(t,x_1,\ldots,x_s,Y) =
g_1(x_1,\ldots,x_s,Y)
g_2(x_1,\ldots,x_s,Y)
\]
is a non-trivial factorisation, so we have indeed reduced the problem to the $s=0$ case.

As $g(T,Y) \in \bZ[T,Y]$ is irreducible, its discriminant in $Y$ is a non-zero polynomial in $\bZ[T]$. Hence, there are $O_d(1)$ integers $t$ such that $g(t,Y) \in \bZ[Y]$ is inseparable. Let us denote by $G$ the Galois group of $g(T,Y)$ over $\bQ(T)$.
Given $t \in \bZ$ such that $g(t,Y)$ is separable, its Galois group $G_t$ is a subgroup of $G$ via the embedding in \cite[Lemma 1]{CD2017}. There are $O_d(1)$ possibilities for $G_t$, and if $g(t,Y)$ is reducible then $G_t$ must be a proper subgroup of $G$, since a polynomial is irreducible if and only if its Galois group is transitive.

Let us now fix a proper subgroup $K$ of $G$, and count integers $t \in [-H,H]$ such that $G_t = K$.
Applying Lemma \ref{GeneralResolvent}, we obtain a polynomial $\Phi(T,Y) \in \bZ[T,Y]$ of total degree $O_D(1)$, monic of degree $[S_d:K]$ in $Y$, such that:
\begin{enumerate}[(i)]
\item If $t \in \bZ \cap [-H,H]$ and $G_t = K$ then $\Phi(t,Y) \in \bZ[Y]$ has an integer root
\[
y \ll_D (|g| H)^{O_{D}(1)}.
\]
\item Any irreducible divisor of $\Phi(T,Y)$ has degree at least $[G:K]$ in $Y$.
\end{enumerate}
Finally, by Theorem \ref{LopsidedBP}, any irreducible divisor of $\Phi(T,Y)$ has $O_{D,\eps}(|g|^\eps H^{\eps+1/2})$ integer roots $(t,y)$ with $|t| \le H$ and $y \ll_D (|g| H)^{O_D(1)}$.
\end{proof}

\subsection{A lopsided version of Pila's theorem}

\begin{thm} \label{LopsidedPila}
Let $g(A,B,Y) \in \bZ[A,B,Y]$ be irreducible of total degree $D$, and monic of degree $d \ge 2$ in the variable $Y$. Then $g$ has $O_{D,\eps}(|g|^\eps H^{1 + \eps + 1/d})$ integer roots $(a,b,y)$ with $|a|,|b| \le H$.
\end{thm}

The bound is essentially sharp, for example if $g(a,b,y) = a - y^d$ then there are at least a positive constant times $H^{1+1/d}$ solutions. Before proceeding towards the proof, we provide some context. By \cite[Lemma 1]{Die2012}, if $|a|, |b| \le H$ and $g(a,b,y) = 0$ then $y \ll_D |g| H^D$. If $g$ were irreducible over $\bR$, then an application of Pila's Theorem \ref{PilaThm} would reveal that $g$ has $O_{D,\eps}(H^{1+\eps+1/d})$ integer roots $(a,b,y)$ with $|a|,|b|,|y| \le H$. Theorem \ref{LopsidedPila} is a variant of this for which $y$ is not constrained to lie in $[-H,H]$.

\begin{proof}
Let $C = C_D$ be a large, positive constant, so that if $|a|,|b| \le H$ and $g(a,b,y) = 0$ then $|y| \le C |g| H^{D}$. We classify $b \in \bZ$ as being:
\begin{itemize}
\item \emph{good}, if $g(A, b, Y) \in \bZ[A,Y]$ is irreducible;
\item \emph{bad}, if $g(A, b, Y) \in \bZ[A,Y]$ is reducible but has no linear divisor;
\item \emph{superbad}, if $g(A, b, Y) \in \bZ[A,Y]$ has a linear divisor.
\end{itemize}

Let $b \in [-H,H]$ be good. In Theorem \ref{LopsidedBP} we have $T \ge (C |g| H^{D})^d$, so there are at most $O(|g|^\eps H^{\eps + 1/d})$ integer zeros $(a,y) \in [-H,H] \times [-C|g| H^D, C |g| H^D]$. Thus, there are at most $O(H^{1 + \eps + 1/d})$ integer triples $(a,b,y) \in [-H,H]^2 \times [-C|g| H^D, C |g| H^D]$ such that $b$ is good and $g(a,b,y) = 0$.

Next, let us suppose instead that $b \in [-H, H]$ is bad. By \cite[Theorem 2.1]{Coh1981}, there exists an integer $a_0 \ll |g|^{O(1)}$ such that $g(a_0,B,Y) \in \bZ[B,Y]$ is irreducible. As $b$ is bad and $g$ is monic in $Y$, the polynomial $g(a_0,b,Y) \in \bZ[Y]$ is reducible, so by Theorem \ref{HIT} there are $O(|g|^\eps H^{\eps+1/2})$ such choices of $b$. Let $h(A,Y) \in \bZ[A,Y]$ be an irreducible divisor of $g(A,b,Y)$, and note from the choice of $b$ that $\deg(h) \ge 2$. By Theorem \ref{LopsidedBP}, there are $O(H^{\eps+1/2})$ zeros $(a,y) \in [-H,H] \times [-C|g| H^D, C |g| H^D]$ of $h$. Whence, there are $O(|g|^\eps H^{1+\eps})$ integer triples $(a,b,y) \in [-H,H]^2 \times \bR$ such that $b$ is bad and $g(a,b,y) = 0$.

Finally, we bound the number $\cN$ of superbad integers $b \in [-H,H]$, and use this to estimate their contribution to the number of roots of $g$ that we are counting. By \cite[Theorem 2.1]{Coh1981}, choose an integer $a_0 \ll |g|^{O(1)}$ such that $g(a_0,B,Y) \in \bZ[B,Y]$ is irreducible. If $b$ is superbad then $g(A,b,Y) \in \bZ[A,Y]$ has a divisor of the form $Y - tA$, for some $t \in \bZ$, and in particular $ g(a_0,b,Y) \in \bZ[Y]$ has an integer root. Thus, by Theorem \ref{LopsidedBP}, we have
\[
\cN \le \# \{ (b,y) \in \bZ^2: |b| \le H, \: |y| \le C|g| H^D, \: g(a_0,b,y) = 0 \}
\ll |g|^\eps H^{\eps+1/d}.
\]
There are $O(H)$ possibilities for $a$, and then there are at most $d$ possibilities for $y$ such that $h(a,b,y) = 0$. Hence, there are $O(|g|^\eps H^{1+\eps+1/d})$ integer triples $(a,b,y) \in [-H,H]^2 \times \bR$ such that $b$ is superbad and $h(a,b,y) = 0$.

We have considered all cases, and conclude that there are $O(|g|^\eps H^{1+\eps+1/d})$ integer triples $(a,b,y) \in [-H,H]^2 \times \bR$ such that $h(a,b,y) = 0$. 
\end{proof}

\subsection{Algebraic toolkit}

In the course of our treatment of higher-degree polynomials, we will work with Galois groups over two-parameter function fields, using the method of Uchida \cite{Uch1970}, Smith \cite{Smi1977} and Cohen \cite{Coh1972, Coh1980}. Here we review some of the theory, following Cohen \cite{Coh1972, Coh1980}. A polynomial is \emph{normal} if its factor type is $(1,\ldots,1)$ or $(1,\ldots,1,2)$. The following criterion is similar to \cite[Lemma 2]{Coh1972}. 

\begin{lemma} [Normality criterion] \label{NormalityCriterion} Let $\cG(X), \cH(X) \in \bC(X) \setminus \{ 0 \}$ be coprime polynomials, and put $r = \cG/\cH$. Then $f := \cG + b\cH$ is normal for all $b \in \bC$ if the following two statements hold:
\begin{itemize}
\item The system 
\begin{equation} \label{FirstSystem}
r'(x) = r''(x) = 0, \qquad \cH(x) \ne 0
\end{equation}
has no solution $x \in \bC$.
\item The system
\begin{equation} \label{SecondSystem}
r(x) - r(y) = r'(x) = r'(y) = 0, \qquad x \ne y, \qquad \cH(x)\cH(y) \ne 0
\end{equation}
has no solution $(x,y) \in \bC^2$.
\end{itemize}
\end{lemma}

\begin{proof} Let $b \in \bC$. First suppose $f$ has a triple root $x \in \bC$. Then
\[
f(x) = f'(x) = f''(x) = 0, \qquad \cH(x) \ne 0,
\]
since $\cG$ and $\cH$ are coprime, and so
\[
\cG(x) = -b \cH(x), \qquad \cG'(x) = -b\cH'(x), \qquad \cG''(x) = -b\cH''(x).
\]
By the quotient rule
\[
r'(x) = \frac{\cH(x) \cG'(x) - \cG(x) \cH'(x)}{\cH(x)^2} = 0
\]
and
\begin{align*}
r''(x) &= \frac{\cH(x)^2 (\cH(x)\cG''(x) - \cG(x)\cH''(x))
- 2\cH(x) \cH'(x)(\cH(x) \cG'(x) - \cG(x) \cH'(x))
}{\cH(x)^4}\\
&= 0.
\end{align*}
In this case $x \in \bC$ is a solution to \eqref{FirstSystem}.

Now suppose $x,y \in \bC$ are distinct double roots of $f$. Then
\[
f(x) = f'(x) = f(y) = f'(y) = 0, \qquad \cH(x)\cH(y) \ne 0,
\]
since $\cG$ and $\cH$ are coprime, and so
\[
\cG(x) = - b\cH(x), \quad \cG'(x) = -b\cH'(x), \quad
\cG(y) = -b\cH(y), \quad \cG'(y) = -b\cH'(y).
\]
By the quotient rule $r'(x) = r'(y) = 0$. Moreover, we have
\[
r(x) - r(y) = -b + b = 0.
\]
In this case $(x,y) \in \bC^2$ is a solution to \eqref{SecondSystem}.
\end{proof}

The following standard consequence of B\'ezout's theorem is a special case of \cite[Lemma~3]{Coh1972}.

\begin{lemma} \label{Bezout} Let $n \in \bN$, and let $f(X,Y), g(X,Y) \in \bC(X,Y)$ be coprime polynomials of degree at most $n$. Then there are $O_n(1)$ solutions $(x,y) \in \bC^2$ to
\[
f(x,y) = g(x,y).
\]
\end{lemma}

If $P(X) \in \bC[X]$ and $Q(X) \in \bC[X] \setminus \{ 0 \}$ are coprime and $R = P/Q$, we write
\[
\deg(R) = \max \{ \deg(P), \deg(Q) \},
\]
as well as
\[
B_R(X,Y) = Q(X) P(Y) - P(X) Q(Y) \in \bC[X,Y]/\sim\]
and
\[
B_R^*(X,Y) = \frac{B_R(X,Y)}{X-Y} \in \bC[X,Y]/\sim,
\]
where the equivalence relation is multiplication by a unit.

\begin{lemma} \label{TechnicalLemma} For $i=1,2$, let $P_i(X) \in \bC[X]$ and $Q_i(X) \in \bC[X] \setminus \{ 0 \}$ be relatively prime, and suppose
\[
x,y \in \bC, \qquad x \ne y, \qquad Q_1(x)Q_1(y)Q_2(x)Q_2(y) \ne 0.
\]
Put $f_i = P_i/Q_i$ for $i=1,2$, and assume that
\[
f_1(x) f_2(y) = f_1(y) f_2(x), \qquad f_2(x)f_2(y) \ne 0.
\]
Then, with $R = f_1/f_2$, we have $B^*_R(x,y) = 0$.
\end{lemma}

\begin{proof} Put 
\[
g = (P_1, P_2), \qquad P_1 = gp_1, \qquad P_2 = gp_2
\]
and 
\[
h = (Q_1,Q_2), \qquad Q_1 = hq_1, \qquad Q_2 = hq_2.
\]
Then
\[
\frac{g(x) g(y) p_1(x) p_2(y)}{h(x) h(y) q_1(x) q_2(y)} =
\frac{g(x) g(y) p_1(y) p_2(x)}{h(x) h(y) q_1(y) q_2(x)}.
\]
As $g(x)g(y) \mid P_2(x)P_2(y)$, we have $g(x)g(y) \ne 0$, so
\[
R(x) = \frac{p_1(x)q_2(x)}{q_1(x)p_2(x)} = \frac{p_1(y) q_2(y)}{q_1(y)p_2(y)} = R(y),
\]
and so $B_R(x,y) = 0$. As $x \ne y$, we finally have $B_R^*(x,y) = 0$.
\end{proof}

The following lemma is essentially \cite[Lemma 4]{Coh1972}.

\begin{lemma} \label{CohenFieldTheory} 
Let $R_1(X) \in \bC(X)$, and suppose $D(X,Y) \in \bC[X,Y]$ is a non-constant divisor of $B_{R_1}^*(X,Y)$. Then there exists $R(X) \in \bC(X)$ with the following properties:
\begin{enumerate}[(i)]
\item If $R_2(X) \in \bC(X)$, then $D(X,Y)$ divides $B_{R_2}^*(X,Y)$ if and only if $R_2(X) \in \bC(R(X))$. \item $\deg(R) \ge 2$. 
\item $R = P/Q$, where $P(X),Q(X) \in \bC[X] \setminus \{ 0 \}$ and $\deg(P) > \deg(Q)$.
\end{enumerate}
\end{lemma}

\begin{proof} Let $\alp_1,\ldots,\alp_u \in \overline{\bC(X)}$ be the roots of $D(X,Y)$, regarded as a polynomial in $Y$ with coefficients in $\bC[X]$, and let $\alp_{u+1},\ldots,\alp_v \in \overline{\bC(X)}$ be the other roots of $B_{R_1}^*(X,Y)$. Set
\[
E = \bC(X, \alp_1, \ldots, \alp_v),
\]
and write $R_1(Y) = f_1(Y)/g_1(Y)$ with $f_1(Y),g_1(Y) \in \bC[Y]$ relatively prime. The polynomial 
\[
f_1(Y) - R_1(X) g_1(Y) \in \bC(R_1(X))[Y]
\]
is irreducible over $\bC(R_1(X))$, and is therefore separable over this perfect field. Therefore $E$, being its splitting field, is a Galois extension of $\bC(R_1(X))$, and moreover $X,\alp_1,\ldots,\alp_v$ are pairwise distinct. Let $G$ be the Galois group of $E$ over $\bC(R_1(X))$, and denote by $S$ the set of $\sig \in G$ such that 
\[
\sig(X) \in \{ X, \alp_1, \ldots, \alp_u \}.
\]
As $S$ contains the stabiliser of $X$, the field $E_1 := E^S$ lies between $\bC$ and $\bC(X) = E^{\mathrm{Stab}(X)}$. Thus, by L\"uroth's theorem \cite[Chapter 4, Theorem 6.8]{Coh1991}, there exists $R(X) \in \bC(X)$ such that $E_1 = \bC(R(X))$. Moreover, as $R_1(X) \in E_1 \setminus \bC$, we know that $R(X)$ is non-constant.

As $f_1(Y) - R_1(X) g_1(Y)$ is irreducible over $\bC(R_1(X))$, its Galois group $G$ acts transitively on its roots
$X, \alp_1,\ldots, \alp_v$. Thus, for each $i \in \{1,2,\ldots,u\}$ there exists $\sig_i \in S$ such that $\sig_i(X) = \alp_i$. Finally, for $R_2(X) \in \bC(X)$, we have
\begin{align*}
D(X,Y) \mid B_{R_2}^*(X,Y) 
&\Leftrightarrow B_{R_2}(X,\alp_i) = 0 \qquad (1 \le i \le u) \\
&\Leftrightarrow R_2(X) = R_2(\alp_i) \qquad (1 \le i \le u) \\
&\Leftrightarrow R_2(X) \in E^S = \bC(R(X)).
\end{align*}

To see that $\deg(R) \ge 2$, observe that $D(X,Y)$ does not divide $B_X^*(X,Y) = 1$. Therefore $X \notin \bC(R(X))$, so $\deg(R) \ne 1$. As $R$ is not constant and $\deg(R) \ne 1$, we must have $\deg(R) \ge 2$.

For the third property, we can swap $P$ and $Q$ if necessary to ensure that $\deg(P) \ge \deg(Q)$. If $\deg(P) = \deg(Q)$ then we can subtract a linear multiple of $P$ from $Q$ to reduce the degree of $Q$ below that of $P$, securing the third property. All of this leaves $\bC(R(X))$ unchanged, so our modified rational function $R$ still has the first two properties.
\end{proof}

We also require the following standard, elementary fact.

\begin{lemma} \label{FullSymmetric} If a group $G$ acts doubly-transitively on $\{ 1, 2, \ldots, n \}$ and contains a transposition, then $G$ is the full symmetric group.
\end{lemma}

\begin{proof} By symmetry, we may suppose that $n \ge 2$ and $(12) \in G$. Let $i,j \in \{1,2,\ldots,n\}$ be distinct. As $G$ acts doubly-transitively, there exists $\sig \in G$ with $\sig(1) = i$ and $\sig(2) = j$. Now $(ij) = \sig (12) \sig^{-1} \in G$.
\end{proof}

\section{Higher-degree polynomials}
\label{HigherSection}

In this section, we establish Theorem \ref{HigherThm}. Let us start by showing that $d \ge 240$, which is the first assertion of the theorem. By \cite[Lemma 3]{Die2012}, if $n \ge 12$ then
\[
d \ge \frac{1}{2} \binom{n}{\lfloor n/2 \rfloor} \ge 462.
\]
We see from the classification by Butler and McKay \cite{BM1983} that if $n \in \{9,11\}$ then $d \ge 240$. In particular, we have secured the inequality $d \ge 240$ in all cases covered by the theorem. It remains to show that
\begin{equation} \label{FinalBound}
N_{G,n} \ll_{n,\eps} H^{n+\eps-3/2+3d^{-1/3}}.
\end{equation}

We use the notation \eqref{fdefGeneral}.

\subsection{Non-degenerate tuples}

We wish to count monic, irreducible polynomials $f$ of degree $n$, with integer coefficients in $[-H,H]$, for which $G_f = G$. Recalling \eqref{PhiDef} and \eqref{ybound}, given such a polynomial $f$, the associated resolvent $\Phi$ has an integer root $y \ll_n H^{O_n(1)}$. To ease notation, we write
\[
a = a_{n-2}, \qquad b = a_{n-1}, \qquad c = a_n,
\]
so that
\begin{equation} \label{fdefabc}
f(X) = X^n + a_1 X^{n-1} + \cdots + a_{n-3} X^3 + aX^2 + bX + c.
\end{equation}

Let us choose 
\[
(a_1, \ldots, a_{n-3}) \in (\bZ \cap [-H,H])^{n-3}\]
non-degenerate, meaning that 
\begin{equation} \label{weird}
-2 a_3 \ne \binom{n}{3} (-2a_1/n)^3 + \binom{n-1}{2}(-2a_1/n)^2 a_1 + (n-2)(-2a_1/n) a_2
\end{equation}
and the vanishing of $g(a, b, c,y) = \Phi(y; a_1,\ldots,a_{n-3}, a, b, c)$ cuts out an absolutely irreducible affine threefold $\cY = \cY_{a_1,\ldots,a_{n-3}}$. By Theorem \ref{HBthm}, there exist 
\[
g_1(a, b, c, y), \ldots, g_J(a,b,c,y) \in \bZ[a,b,c,y],
\]
coprime to $g(a,b,c,y)$ and of degree $O_{n,\eps}(1)$, where $J \ll H^{\eps + 3d^{-1/3}}$, such that if $(a,b,c,y) \in \cY$ and
\[
|a|,|b|,|c| \le H, \qquad y \ll_n H^{O_n(1)}
\]
then
\begin{equation} \label{surface}
g(a,b,c,y) = g_j(a,b,c,y) = 0
\end{equation}
for some $j$. Now we fix $j$ and count solutions to \eqref{surface}.

Note that $\deg_y(g) = d > 0$. If $\deg_y(g_j) = 0$ then let $F(a,b,c) = g_j(a,b,c,y)$, and otherwise let $F(a,b,c)$ be the resultant of $g$ and $g_j$ in the variable $y$.
By \cite[Chapter 3, \S 6, Proposition 3]{CLO2015}, applied with $k = \bQ(a,b,c)$, this is a non-zero element of $\bZ[a,b,c]$. By \cite[Chapter 3, \S 6, Proposition 5]{CLO2015}, we have $F(a,b,c) = 0$ for any solution $(a,b,c,y)$ to \eqref{surface}.

Next, we factorise over the reals. For $a,b,c \in \bZ$, observe that $F(a,b,c) = 0$ if and only if we have $\cF(a,b,c) = 0$ for some irreducible divisor $\cF(a,b,c) \in \bR[a,b,c]$ of $F(a,b,c)$. By Theorem \ref{PilaThm}, if $\cF$ is non-linear then
\[
\# \{ (a,b,c) \in (\bZ \cap [-H,H])^3: \cF(a,b,c) = 0 \} \ll H^{\eps + 3/2}.
\]
Then $y$ is determined from $g(a,b,c,y) = 0$ in at most $d$ ways, so the number of solutions $(a,b,c,y)$ counted in this case is $O(H^{\eps + 3/2})$. The total contribution to $N_{G,n}$ from this non-linear case is therefore
\begin{equation} \label{DominantContribution}
O(H^{n-3} H^{\eps + 3d^{-1/3}}  H^{\eps + 3/2})
\ll H^{2\eps + n -3/2 + 3d^{-1/3}}.
\end{equation}

\bigskip

Next, suppose $\cF$ is linear. Then 
\[
\cF(a,b,c) = \alp a + \bet b + \gam c + \del = 0,
\]
for some $\alp, \bet, \gam, \del \in \bR$ with $(\alp, \bet, \gam) \ne (0,0,0)$. Supposing for the time being that $\cF$ is not a multiple of a rational polynomial, we can write
\[
\cF = \sum_{i \le r} \lam_i F_i,
\]
where $r \in \{ 2, 3, 4 \}$, the $F_i$ are linear polynomials with rational coefficients, and $\lam_1,\ldots,\lam_r$ are linearly independent over $\bQ$. Following the argument presented in the proof of \cite[Corollary 1]{HB2002}, some $F_i$ is not a multiple of $\cF$, and any rational root of $\cF$ must also be a root of this $F_i$. It then follows from linear algebra that there are $O(H)$ possibilities for $(a,b,c)$. The contribution to $N_{G,n}$ from this situation is $O(H^{n-2+\eps+3d^{-1/3}})$.

Thus, we may assume in the sequel that $\alp, \bet, \gam, \del \in \bZ$. Moreover, if $\cF(a,b,c) = 0$ has an integer solution then $(\alp,\bet,\gam) \mid \del$, so we may divide through by $(\alp,\bet,\gam)$ and assume that $(\alp,\bet,\gam) = 1$.

\begin{lemma} If $\max\{ |\alp|, |\bet|, |\gam| \} > H$ then
\[
\# \{ (a,b,c) \in (\bZ \cap [-H,H]^3): \alp a + \bet b + \gam c + \del = 0 \} \ll H.
\]
\end{lemma}

\begin{proof} By symmetry, we may assume that $\gam > H$. We may also assume that there exists $(a_0,b_0,c_0) \in (\bZ \cap [-H,H])^3$ for which 
\[
\alp a_0 + \bet b_0 + \gam c_0 + \del = 0.
\]
Now
\[
\alp x + \bet y + \gam z = 0,
\]
where $x = a-a_0$, $y = b-b_0$ and $z = c - c_0$ lie in $[-2H,2H]$, and in particular
\[
\alp x + \bet y \equiv 0 \mmod \gam.
\]
This condition defines a full-rank lattice $\Lam$ of determinant $\gam$, by \cite[Lemma 2.14]{CT2021}. Its first successive minimum satisfies $\lam_1 \ge 1$, since $\Lam \subset \bZ^2$. Thus, by \cite[Lemma 2]{Sch1995}, we have
\[
\# \{ (x,y) \in \Lam \cap [-2H,2H]^2 \}
\ll \frac{H^2}\gam + H \ll H.
\] 
\end{proof}

The contribution to $N_{G,n}$ from the scenario $\max\{|\alp|,|\bet|, |\gam|\} > H$ is therefore 
\[
O(H^{n-2+\eps+3d^{-1/3}}),
\]
which is negligible for the purposes of proving \eqref{FinalBound}. We may therefore suppose that $|\alp|,|\bet|,|\gam| \le H$. Now there are no solutions $(a,b,c) \in [-H,H]^3$ to
\[
\alp a + \bet b + \gam c + \del = 0
\]
unless $|\del| \le 3H^2$.
Henceforth, we assume that 
\begin{equation} \label{CoefficientBound}
|\alp|, |\bet|, |\gam| \le H,
\qquad |\del| \le 3H^2.
\end{equation}

\bigskip

Polynomials $G_0(X), G_1(X), G_2(X) \in \bC[X]$ are \emph{strongly totally composite} if there exist a rational function $R = P/Q$, where $P(X) \in \bC[X]$ and $Q(X) \in \bC[X] \setminus \{ 0 \}$ are coprime, and $P_0(X), P_1(X), P_2(X) \in \bC[X]$, such that
\[
\deg(R) \ge 2, \qquad \deg(P) > \deg(Q), \qquad 
L := \max \{ \deg(P_0), \deg(P_1), \deg(P_2) \} \ge 2,
\]
and
\[
G_i(X)= Q(X)^L P_i(R(X)) \qquad (i=0,1,2).
\]

The analysis naturally divides into three cases. The overall structure of the argument will be the same in the three cases, but some details will differ.

\subsubsection{Case 1: $\gam \ne 0$} 

Then
\begin{align*}
f(X) &= X^n + a_1 X^{n-1} + \cdots + a_{n-3} X^3
+ aX^2 + bX - \gam^{-1} (\alp a + \bet b + \del) \\
&= G_0(X) + a G_1(X) + b G_2(X),
\end{align*}
where
\[
G_0(X) = X^n + a_1 X^{n-1} + \cdots + a_{n-3} X^3 - \del/\gam,
\quad
G_1(X) = X^2 - \alp/\gam, 
\quad
G_2(X) = X - \bet/\gam.
\]
We study $\Gal(f, \bC(a,b))$ using the method of Uchida \cite{Uch1970}, Smith \cite{Smi1977} and Cohen \cite{Coh1972, Coh1980}. As we only wish to count irreducible polynomials, we may assume that there exists an irreducible specialisation over the rationals. Whence $f$ is irreducible over $\bQ(a,b)$, and is therefore also separable. Note that $G_0,G_1,G_2$ are coprime, and that they are linearly independent over $\bC$. 

Recall the following elementary fact.

\begin{lemma}
Let $n \ge 3$ be an integer, and let $G$ be a group. Then the action of $G$ on $\{1,2,\ldots,n\}$ is doubly-transitive if and only if the stabiliser subgroup of any point acts transitively on the remaining points. 
\end{lemma}

\begin{proof} The forward direction is clear. For the backward direction, let $i_1 \ne i_2$ and $j_1 \ne j_2$ be elements of $\{1,2,\ldots,n\}$, and let $k \in \{1,2,\ldots,n\} \setminus \{i_1,j_1\}$. Then the composition
\begin{align*}
i_1 &\mapsto i_1 \mapsto j_1 \mapsto j_1 \\
i_2 &\mapsto k \mapsto k \mapsto j_2
\end{align*}
maps $(i_1,i_2)$ to $(j_1,j_2)$.
\end{proof}

In light of this, the Galois group of $f$ is doubly-transitive if and only if $f(X)/(X-x)$ is irreducible over $\bC(a,b,x)$ for any root $x \in \overline{\bC(a,b)}$. Note that any root of $f$ must be transcendental over $\bC$.

\begin{lemma} \label{DoublyTransitive} The permutation group $\Gal(f, \bC(a,b))$ is doubly-transitive.
\end{lemma}

\noindent This is analogous to \cite[Lemma 4]{Coh1980}.

\begin{proof}
Assume for a contradiction that $\Gal(f, \bC(a,b))$ is not doubly-transitive. Then there exists a root $x \in \overline{\bC(a,b)}$ of $f$ such that $f(X)/(X-x)$ is reducible over $\bC(a,b,x)$. By Gauss's lemma, the polynomial $f(X)/(X-x)$ is reducible over $\bC[a,b,x]$. As $f(x) = 0$, we have 
\[
bG_2(x) = -G_0(x) - aG_1(x),
\]
so
\begin{align*}
G_2(x) f(X) &= G_2(x) (G_0(X) + a G_1(X) + b G_2(X)) \\
&= G_2(x)G_0(X) - G_0(x) G_2(X) + a(G_2(x) G_1(X) - G_1(x) G_2(X)).
\end{align*}
This is linear in $a$ and separable in $X$, so for $f(X)/(X-x)$ to be reducible there must exist $\xi \in \overline{\bC(x)} \setminus \{ x \}$ for which
\[
G_2(x) G_i (\xi) = G_2(\xi) G_i(x) \qquad (i=0,1).
\]

Note that $x \ne \bet/\gam$, being transcendental over $\bC$. The equation with $i=1$ is quadratic in $\xi$, with one of the roots being $x$. Therefore $\xi$ equals the other root:
\[
\xi = 
\frac{x^2 - \alp/\gam}{x - \bet/\gam} \: - x = 
\frac{\bet x - \alp} { \gam x - \bet}.
\]
Substituting this into the equation with $i=0$ yields
\[
\left(x - \frac \bet \gam \right) G_0 \left( \frac{\bet x - \alp} { \gam x - \bet} \right) = \left( \frac{\bet x - \alp} { \gam x - \bet} - \frac \bet \gam \right) G_0(x) = \frac{\bet^2 - \alp \gam}{\gam(\gam x - \bet)} G_0(x),
\]
or equivalently
\[
(\gam x - \bet)^n G_0 \left( \frac{\bet x - \alp} { \gam x - \bet} \right) - (\bet^2 - \alp \gam) (\gam x - \bet)^{n-2} G_0(x) = 0.
\]
As $x$ is transcendental over $\bC$, it cannot be the root of a non-zero polynomial over $\bC$, and whence
\[
(\gam X - \bet)^n G_0 \left( \frac{\bet X - \alp} { \gam X - \bet} \right) - (\bet^2 - \alp \gam) (\gam X - \bet)^{n-2} G_0(X) = 0 \in \bC[X],
\]
or equivalently
\[
G_0 \left( \frac{\bet X - \alp} { \gam X - \bet} \right) = \frac{(\bet^2 - \alp \gam) G_0(X)}{(\gam X - \bet)^2} \in \bC(X).
\]

Taking $X \to \infty$ yields $\bet^2 - \alp \gam = 0$ and $G_0(\bet/\gam) = 0$. As 
\[
G_1(\bet/\gam) = \bet^2/\gam^2-\alp/\gam = \gam^{-2} (\bet^2 - \alp \gam) = 0,\qquad G_2(\bet/\gam) = 0,
\]
we now have $f(\bet/\gam) = 0$, contradicting the irreducibility of $f$ over $\bQ(a,b)$.
\end{proof}

We will show that the permutation group $\mathrm{Gal}(f,\bC(a,b))$ contains a transposition, but first we require a preparatory result.

\begin{lemma} \label{NotTotallyComposite} The polynomials $G_0, G_1, G_2$ are not strongly totally composite.
\end{lemma}

\begin{proof} Suppose for a contradiction that
\[
G_2(X) = Q(X)^L P_2(R(X)), \qquad R=P/Q,
\]
where $P(X), Q(X) \in \bC[X] \setminus \{ 0 \}$ are coprime, $P_2(X) \in \bC[X]$, and 
\[
\deg(P) > \max \{ \deg(Q), 1 \}, \qquad L \ge s := \deg(P_2), \qquad L \ge 2.
\]
Writing
\[
P_2(X) = b_s X^s + \cdots + b_0, \qquad b_0, \ldots, b_s \in \bC, \qquad b_s \ne 0,
\]
we have
\[
X - \bet/\gam = Q(X)^{L-s} (b_s P(X)^s + \cdots + b_0 Q(X)^s).
\]
The right hand side has degree $(L-s) \deg(Q) + s \cdot \deg(P) \ge (L-s) \deg(Q) + 2s$. Now $s = 0$ and $\deg(Q) = 0$, and we reach a contradiction, completing the proof.
\end{proof}

\begin{lemma} \label{transposition} The permutation group $\mathrm{Gal}(f,\bC(a,b))$ contains a transposition.
\end{lemma}

This is analogous to \cite[Corollary 6 and Lemma 7]{Coh1980}. 

\begin{proof}
By \cite[Corollary 6]{Coh1980}, it suffices to establish the existence of $a,b \in \bC$ for which $f$ has factor type $(1,\ldots,1,2)$. We proceed in two steps. The first is to show that, for all but finitely many $a \in \bC$, there exists $b \in \bC$ such that $\Del = 0$. To achieve this, we need to demonstrate the existence of $x,b \in \bC$ such that 
\[
G_0(x) + aG_1(x) + bG_2(x) = 0, \qquad
G_0'(x) + aG_1'(x) + b = 0.
\]
We choose $b = - (G_0'(x) + aG_1'(x))$, and now we need to solve
\[
G_0(x) + aG_1(x) - (G_0'(x) + aG_1'(x)) G_2(x) = 0.
\]
As $G_1(X) - G_1'(X) G_2(X) = X^2 - \alp/\gam - 2X(X-\bet/\gam)$ is non-constant, there can be at most one value of $a$ for which 
\[
G_0(X) + aG_1(X) - (G_0'(X) + aG_1'(X)) G_2(X)
\]
is constant. For any other value of $a \in \bC$, there must exist a solution $x \in \bC$. This completes the first step.

\bigskip

The remaining second step is to show that, for all but finitely many $a \in \bC$, for any $b \in \bC$ the polynomial $f$ is normal. For this we will apply Lemma \ref{NormalityCriterion} with $\cG = G_0 + aG_1$ and $\cH = G_2$. As
\[
\cG\left(
\frac{\bet}{\gam} \right) = G_0 \left( \frac{\bet}{\gam} \right) + a \left( \frac{\bet^2}{\gam^2} \: - \: \frac{\alp}{\gam} \right)
\]
is a non-zero polynomial of degree at most 1 in the variable $a$, there is at most one value of $a \in \bC$ such that $\cG$ and $\cH$ are not coprime over $\bC$.

Suppose $a,x \in \bC$ and that we have \eqref{FirstSystem}, where $r = \cG/\cH$. Then
\begin{align*}
0 &= \cH(x)^2 r'(x) = \cH(x) \cG'(x) - \cG(x) \cH'(x) \\
&= (x-\bet/\gam) (G'_0(x) + 2xa) - (G_0(x) + a(x^2 - \alp/\gam))
\end{align*}
and
\begin{align*}
0 &= \cH(x)^3 r''(x) = 
\cH(x)(\cH(x)\cG''(x) - \cG(x)\cH''(x))
- 2\cH'(x)(\cH(x) \cG'(x) - \cG(x) \cH'(x)) \\
&= (x - \bet/\gam)^2 (G_0''(x) +2a) 
- 2((x - \bet/\gam) (G_0'(x) +2xa) - (G_0(x) + a (x^2 - \alp/\gam))),
\end{align*}
so
\[
(x - \bet/\gam) (G''_0(x)+2a) = 0.
\]
By \eqref{FirstSystem}, we have $x \ne \bet/\gam$. Hence $2a = -G''_0(x)$, and so
\[
2(x-\bet/\gam)(G_0'(x) - G''_0(x)x) - 2G_0(x) + G_0''(x)(x^2 - \alp/\gam) = 0.
\]
The polynomial has degree exactly $n$, so there are at most $n$ solutions $(a,x) \in \bC^2$ to \eqref{FirstSystem}. 

\bigskip

Now suppose instead that we have \eqref{SecondSystem} for some $a,x,y \in \bC$. With 
\[
r_0 = G_0/\cH, \qquad r_1 = G_1/\cH,
\]
we have
\[
M \begin{pmatrix} 1 \\ a \end{pmatrix} = 
\begin{pmatrix} 0 \\ 0 \\ 0 \end{pmatrix},
\]
where
\[
M = \begin{pmatrix}
r_0(x) - r_0(y) & r_1(x) - r_1(y) \\
r_0'(x) & r_1'(x)\\
r_0'(y) & r_1'(y)
\end{pmatrix}.
\]

Assume for a contradiction that $\rank(M) = 0$. Then
\[
2x = \frac{x^2 - \alp/\gam}{x-\bet/\gam} = \frac{y^2 - \alp/\gam}{y-\bet/\gam} = 2y,
\]
contradicting \eqref{SecondSystem}. The upshot is that $\rank(M) \ge 1$.

There are no solutions $a \in \bC$ if $\rank(M) = 2$, so we may assume in the sequel that $\rank(M) = 1$, in which case there is at most one solution $a \in \bC$. Next, we use an argument from \cite{Coh1972}. Considering minors yields
\[
r_0'(x) r_1'(y) = r_1'(x) r_0'(y)
\]
and
\[
r_0'(x) (r_1(x) - r_1(y)) = r_1'(x) (r_0(x) - r_0(y)),
\]
as well as 
\[
r_0'(y)(r_1(x)-r_1(y)) = r_1'(y) (r_0(x) - r_0(y)).
\]

\bigskip

\textbf{Solutions with $r_0'(x)r_0'(y) = 0$.} By symmetry, it suffices to consider solutions with $r_0'(y) = 0$. As 
\[
\cH(y)^2 r_0'(y) = \cH(y)G_0'(y) - G_0(y)
\]
has degree $n$ in $y$, there are only finitely many possibilities for $y$. Then
\begin{align*}
&\cH(x)^3 \cH(y) (r_0'(x) (r_1(x) - r_1(y)) - r_1'(x) (r_0(x) - r_0(y)))\\
&=  
(\cH(x)G_0'(x)-G_0(x))(\cH(y)G_1(x) - G_1(y) \cH(x)) \\ &\qquad
-(\cH(x)G_1'(x)-G_1(x))(\cH(y)G_0(x) - G_0(y) \cH(x))
\end{align*}
is a degree $n+2$ polynomial in $x$, so it has finitely many zeros. We conclude that there are finitely many solutions $(a,x,y)$ of this type.

\bigskip

\textbf{Solutions with $r_0'(x)r_0'(y) \ne 0$.}
We have
\[
r_0'(x)r_1(x) - r_0(x) r_1'(x) = r_0'(x) r_1(y) - r_0(y) r_1'(x),
\]
and so
\begin{align*}
r_0'(y)(r_0'(x)r_1(x) - r_0(x) r_1'(x))
&= r_0'(y) (r_0'(x) r_1(y) - r_0(y) r_1'(x)) \\
&= r_0'(x)(r_0'(y) r_1(y) - r_0(y) r_1'(y)).
\end{align*}
Thus, by Lemma \ref{TechnicalLemma}, we have
\[
B_{R_1}^*(x,y) = B_{R_2}^*(x,y) = 0, 
\]
where 
\[
R_1 = \frac{r_1'}{r_0'}, \qquad 
R_2 =  \frac{r_0 r_1' - r_0' r_1}{r_0'}.
\]

By Lemma \ref{Bezout}, this has at most finitely many solutions unless the polynomials $B_{R_1}^*(X,Y)$ and $B_{R_2}^*(X,Y)$ have a common divisor of positive degree. Let us now assume the latter, for a contradiction. By Lemma \ref{CohenFieldTheory}, there exist $R(X), \tet(X), \phi(X) \in \bC(X)$  such that $\deg(R) \ge 2$ and
\begin{equation} \label{TwoEquations}
r_1'(X) = \tet(R(X)) r_0'(X),
\qquad r_0(X) r_1'(X) - r_0' (X) r_1(X) =  \phi (R(X)) r_0'(X),
\end{equation}
and furthermore
\[
R = P/Q, \qquad P(X), Q(X) \in \bC[X] \setminus \{ 0 \}, 
\qquad \deg(P) > \deg(Q),
\]
and $P,Q$ are coprime.

Assume for a contradiction that $\tet$ is constant. Then, by the quotient rule, we have
\[
\cH G_1' - G_1 \cH' = \lam (\cH G_0' - G_0 \cH')
\]
for some $\lam \in \bC$. More explicitly, we have
\[
2X(X-\bet/\gam) - (X^2 - \alp/\gam) = \lam ((X - \bet/\gam) G_0'(X) - G_0(X)).
\]
Equating $X^n$ coefficients yields $\lam = 0$, and then equating $X^2$ coefficients delivers a contradiction.  Therefore $\tet$ is non-constant.

Next, we substitute the first equation of \eqref{TwoEquations} into the second. As $r_0'(X) \ne 0$, this gives
\[
\tet(R(X)) r_0(X) - r_1(X) = \phi (R(X)).
\]
Differentiating yields
\[
r_1'(X) = \tet'(R(X)) R'(X) r_0(X) + \tet(R(X)) r_0'(X) - \phi'(R(X)) R'(X).
\]
Since $R', \tet' \ne 0$ and $r_1'(X) = \tet(R(X)) r_0'(X)$, we now have
\[
r_0(X) = \frac{\phi'(R(X))} {\tet'(R(X))}.
\]

We had observed that there is at most one value of $a \in \bC$ for which $\cG$ and $\cH$ are not coprime over $\bC$. Let us now suppose that $a \in \bC$ takes any other value. From the previous paragraph, we see that $r_0(X)$, and hence also $r_0(X) + a r_1(X)$, is a rational function of $R(X)$. Thus, for some coprime polynomials $g_0(X), g_1(X) \in \bC[X]$, we have
\[
\frac{\cG(X)}{\cH(X)} = \frac{g_0(R(X))}{g_1(R(X))}.
\]
Put
\[
L = \max \{ \deg (g_0), \deg(g_1) \} \ge 1,
\]
and observe that
\[
\frac{\cG(X)}{\cH(X)} = \frac{Q(X)^L g_0(R(X))}{Q(X)^L g_1(R(X))}.
\]

Suppose for a contradiction that the polynomials $Q(X)^L g_0(R(X))$ and $Q(X)^L g_1(R(X))$ are not coprime. Then some $x \in \bC$ is a common root. If $Q(x) \ne 0$ then $R(x)$ is a common root of $g_0$ and $g_1$, which is impossible because $g_0$ and $g_1$ are coprime. Therefore $Q(x) = 0$. Let $i \in \{ 0, 1\}$ be such that $L = \deg(g_i)$, and write
\[
g_i(X) = b_L X^L + b_{L-1} X^{L-1} + \cdots + b_0, \qquad b_0,\ldots,b_L \in \bC, \qquad b_L \ne 0.
\]
Then 
\[
0 = Q(x)^L g_i(P(x)/Q(x)) = b_L P(x)^L,
\]
so $P(x)=0$, contradicting the coprimality of $P$ and $Q$. Hence $Q(X)^L g_0(R(X))$ and $Q(X)^L g_1(R(X))$ are coprime. 

Now
\[
\cG(X) = \tau Q(X)^L g_0(R(X)), \qquad \cH(X) =  \tau Q(X)^L g_1(R(X)),
\]
for some $\tau \in \bC \setminus \{ 0\}$. We also know that $r_0(X)$ and $r_1(X)$ are rational functions of $R(X)$. Therefore
\[
G_0(X) = Q(X)^L g_2(R(X)), \qquad G_1(X) = Q(X)^L g_3(R(X)),
\]
for some rational functions $g_2, g_3$. 

Let $g_2 = P_2/Q_2$, where $P_2(X), Q_2(X) \in \bC[X]$ are coprime and $Q_2 \ne 0$, and let 
\[
L_2 = \deg(g_2).
\]
Write
\[
Q_2(X) = c_{s} X^{s} + c_{s-1} X^{s-1} + \cdots + c_0, \qquad c_s, \ldots, c_0 \in \bC, \qquad c_s \ne 0.
\]

Assume for a contradiction that $s > 0$. Arguing as before, the polynomial $Q_2(R(X)) Q(X)^{L_2}$ is coprime to $P_2(R(X)) Q(X)^{L_2}$. Therefore $Q_2(R(X)) Q(X)^{L_2}$ divides $Q(X)^L$ in $\bC[X]$. Whence $L \ge L_2 - s$, and
\[
c_s P(X)^s + c_{s-1} P(X)^{s-1} Q(X) + \cdots + c_0 Q(X)^s = Q_2(R(X)) Q(X)^s \mid Q(X)^{L - L_2 + s}.
\]
The polynomial $Q_2(R(X)) Q(X)^s$ has degree $s \cdot \deg(P) > 0$. Therefore
$L > L_2 - s$ and, for some $x \in \bC$, we have
\[
c_s P(x)^s + c_{s-1} P(x)^{s-1} Q(x) + \cdots + c_0 Q(x)^s = 0,
\]
so $Q(x)^{L-L_2+s} = 0$. Now $Q(x) = 0$, which implies that $c_sP(x)^s = 0$ and so $P(x) = 0$, this being impossible because $P$ and $Q$ are coprime. 

Hence $s = 0$. The upshot is that $g_2$ is a polynomial, and similarly $g_3$ is a polynomial. Recall that
\[
G_0(X) = Q(X)^L g_2(R(X)), \qquad G_1(X) = Q(X)^L g_3(R(X))
\]
and
\[
G_2(X) = \cH(x) =  Q(X)^L \cdot \tau g_1(R(X)), 
\]
where $\tau \in \bC \setminus \{ 0 \}$. As
\[
G_0(X) + aG_1(X) = Q(X)^L g_0(R(X)),
\]
we deduce that
\[
\deg(g_0) = \deg(g_2) > \deg(g_3),
\]
whereupon
\[
L = \max \{ \deg(g_1), \deg(g_2), \deg(g_3) \},
\]
and finally
\[
L \ge 2
\]
by the linear independence of $G_0, G_1, G_2$. This contradicts Lemma \ref{NotTotallyComposite}, that $G_0,G_1,G_2$ are not strongly totally composite.

We conclude that, for all but finitely many values of $a \in \bC$, the systems \eqref{FirstSystem} and \eqref{SecondSystem} each have no solutions $x,y \in \bC$. Thus, by Lemma \ref{NormalityCriterion}, for any $b \in \bC$ the polynomial $f$ is normal. This completes the second step, and with it the proof of the lemma.
\end{proof}

By Lemmas \ref{FullSymmetric}, \ref{DoublyTransitive}, and \ref{transposition}, we have $\Gal(f, \bC(a,b)) = S_n$. As $\Gal(f, \bQ(a,b))$ contains $\Gal(f, \bC(a,b))$, we conclude that
\[
\Gal(f, \bQ(a,b)) = S_n.
\]
Let $h$ be as in Lemma \ref{hdef}, and note that $\deg(h) \ll_n 1$. By Lemma \ref{ResolventProperty}, if $G_f = G$ then $h(a,b,y) = 0$ for some $y \in \bZ$. This has $O(H^{1+\eps+1/d})$ integer roots $(a,b,y)$ with $|a|,|b| \le H$, by \eqref{CoefficientBound} and Theorem \ref{LopsidedPila}. In summary, this case $\gam \ne 0$ contributes $O(H^{\eps + n - 2 + 3d^{-1/3} + d^{-1}})$ to $N_{G,n}$.

\subsubsection{Case 2: $\gam = 0$ and $\bet \ne 0$} Then
\begin{align*}
f(X) &= X^n + a_1 X^{n-1} + \cdots + a_{n-3} X^3
+ aX^2 - \bet^{-1}(\alp a + \del) X + c\\
&= G_0(X) + a G_1(X) + cG_2(X),
\end{align*}
where
\[
G_0(X) = X^n + a_1 X^{n-1} + \cdots + a_{n-3} X^3 - \: \frac{\del}{\bet}X,
\quad
G_1(X) = X^2 - \: \frac{\alp}{\bet}X, 
\quad
G_2(X) = 1.
\]
We may assume that $f$ is irreducible over $\bQ(a,c)$, and therefore also separable. 

\begin{lemma} \label{DoublyTransitive2} The permutation group $\Gal(f, \bC(a,c))$ is doubly-transitive.
\end{lemma}

\begin{proof}
We imitate the proof of Lemma \ref{DoublyTransitive}. Assume for a contradiction that $\Gal(f, \bC(a,c))$ is not doubly-transitive. Then there exists a root $x \in \overline{\bC(a,b)}$ of $f$ such that $f(X)/(X-x)$ is reducible over $\bC[a,b,x]$. Now 
\[
cG_2(x) = -G_0(x) - aG_1(x),
\]
so
\[
G_2(x) f(X) = G_2(x)G_0(X) - G_0(x) G_2(X) + a(G_2(x) G_1(X) - G_1(x) G_2(X)).
\]
This is linear in $a$ and separable in $X$, so for $f(X)/(X-x)$ to be reducible there must exist $\xi \in \overline{\bC(x)} \setminus \{ x \}$ such that
\[
G_2(x) G_i (\xi) = G_2(\xi) G_i(x) \qquad (i=0,1).
\]

The equation with $i=1$, namely
\[
\xi^2 - \: \frac{\alp}{\bet} \xi = x^2 - \: \frac{\alp}{\bet}x,
\]
is quadratic in $\xi$ with one of the roots being $x$. Therefore $\xi$ equals the other root:
\[
\xi = \frac{\alp}{\bet} - x.
\]
Substituting this into the equation with $i=0$ yields
\[
G_0 \left( \frac{\alp}{\bet} - x \right) = G_0(x).
\]
As $x$ is transcendental over $\bC$, it cannot be the root of a non-zero polynomial over $\bC$, and whence
\begin{equation} \label{reflection}
G_0 \left( \frac{\alp}{\bet} \: - X\right) = G_0(X).
\end{equation}

Equating $X^n$ coefficients tells us that $n$ is even. Then, equating $X^{n-1}$ coefficients gives
\[
-n \alp/ \bet -a_1 = a_1,
\]
so $\alp/\bet = -2a_1/n$. Finally, equating $X^{n-3}$ coefficients yields
\begin{align*}
-2a_3 &= \binom{n}{3} (\alp/\bet)^3 + \binom{n-1}{2}(\alp/\bet)^2 a_1 + (n-2)(\alp/\bet)a_2 \\
&= \binom{n}{3} (-2a_1/n)^3 + \binom{n-1}{2}(-2a_1/n)^2 a_1 + (n-2)(-2a_1/n) a_2, 
\end{align*}
contradicting \eqref{weird} and completing the proof.
\end{proof}

\begin{lemma} \label{NotTotallyComposite2} The polynomials $G_0, G_1, G_2$ are not strongly totally composite.
\end{lemma}

\begin{proof} Suppose for a contradiction that
\[
G_i(X) = Q(X)^L P_i(R(X)) \qquad (i=0,1,2), \qquad R=P/Q,
\]
where $P(X), Q(X) \in \bC[X] \setminus \{ 0 \}$ are coprime, 
\[
P_0(X), P_1(X), P_2(X) \in \bC[X], \qquad s_1 := \deg(P_1),
\qquad 
s_2 := \deg(P_2), 
\]
and 
\[
\deg(P) > \max \{ \deg(Q), 1 \}, \qquad L \ge  \max \{ s_1, s_2, 2 \}.
\]
Writing
\[
P_2(X) = b_{s_2} X^{s_2} + \cdots + b_0, \qquad b_0, \ldots, b_{s_2} \in \bC, \qquad b_{s_2} \ne 0
\]
and
\[
P_1(X) = c_{s_1} X^{s_1} + \cdots + c_0, \qquad c_0, \ldots, c_{s_1} \in \bC, \qquad c_{s_1} \ne 0
\]
we have
\[
1 = Q(X)^{L-{s_2}} (b_{s_2} P(X)^{s_2}+ \cdots + b_0 Q(X)^{s_2})
\]
and
\[
X(X - \alp/\bet) = Q(X)^{L-{s_1}} (c_{s_1} P(X)^{s_1} + \cdots + c_0 Q(X)^{s_1}).
\]
Therefore
\[
0 = (L-s_2) \deg(Q) + s_2 \cdot \deg(P)
\]
and
\[
2 = (L-s_1) \deg(Q) + s_1 \cdot \deg(P).
\]
The only possibility is
\[
s_2 = 0, \qquad \deg(Q) = 0, \qquad s_1 = 1, \qquad \deg(P) = 2.
\]

Now, for some $\lam_0 \in \bC$, we have
\[
X(X-\alp/\bet) = \lam_0^{L-1} (c_1 P(X) + c_0 \lam_0).
\]
As $\lam_0, c_1 \ne 0$, this enables us to write $G_0$ as a polynomial in $X(X-\alp/\bet)$. In particular, we have \eqref{reflection}, and this leads to a contradiction in the same way as in Case 1.
\end{proof}

\begin{lemma} \label{transposition2} The permutation group $\mathrm{Gal}(f,\bC(a,c))$ contains a transposition.
\end{lemma}

\begin{proof}
Following the proof of Lemma \ref{transposition}, the first step is to show that if $a \in \bC$ then there exist $x,c \in \bC$ such that 
\[
G_0(x) + aG_1(x) + c = 0, \qquad
G_0'(x) + aG_1'(x) = 0.
\]
The second equation is the vanishing of a non-zero polynomial, so it has a solution $x$, and then the first equation has a solution $c$. The rest of the proof is essentially the same as in Case 1.
\end{proof}

By Lemmas \ref{FullSymmetric}, \ref{DoublyTransitive2}, and \ref{transposition2}, we have $\Gal(f, \bC(a,c)) = S_n$. Arguing as in the previous case, we find that this case, too, contributes $O(H^{\eps + n - 2 + 3d^{-1/3} + d^{-1}})$ to $N_{G,n}$.

\subsubsection{Case 3: $\bet = \gam = 0$ and $\alp \ne 0$} Then
\begin{align*}
f(X) &= X^n + a_1 X^{n-1} + \cdots + a_{n-3} X^3
- \alp^{-1} \del X^2 + bX + c\\
&= G_0(X) + b G_1(X) + cG_2(X),
\end{align*}
where
\[
G_0(X) = X^n + a_1 X^{n-1} + \cdots + a_{n-3} X^3 - \alp^{-1} \del X^2,
\quad
G_1(X) = X, 
\quad
G_2(X) = 1.
\]
Similarly to the previous cases, or directly from \cite[Theorem  1]{Coh1980}, we have $\Gal(f, \bC(b,c)) = S_n$. Arguing as in the previous cases, we find that this case also contributes $O(H^{\eps + n - 2 + 3d^{-1/3} + d^{-1}})$ to $N_{G,n}$.

\bigskip

We conclude that non-degenerate tuples contribute $O(H^{\eps + n - 3/2 + 3d^{-1/3}})$ to $N_{G,n}$, the dominant contribution coming from \eqref{DominantContribution}.

\subsection{Degenerate tuples} \label{DegenerateTuples}

We begin by counting degenerate tuples. There are $O(H^{n-4})$ integer vectors $(a_1,\ldots,a_{n-3}) \in [-H,H]^{n-3}$ \textbf{not} satisfying \eqref{weird}.

For the degenerate tuples with \eqref{weird}, the idea is to show that $\cY_{a_1, \ldots, a_{n-3}}$ is absolutely irreducible for generic $a_1, \ldots, a_{n-3}$. Let
\[
\cD(a_1,\ldots,a_n) \in \bZ[a_1,\ldots,a_n]
\]
be the discriminant of $\Phi(y)$. As $\Phi$ has total degree $O_n(1)$, so too does $\cD$. Further, we know from \cite[Lemma 7]{Die2012} that $\cD$ is not the zero polynomial. Whence, for all but $O(H^{n-4})$ integer vectors $(a_1,\ldots,a_{n-3}) \in [-H,H]^{n-3}$, the polynomial $g(a,b,c,y) = \Phi(y;a_1,\ldots,a_{n-3},a,b,c)$ is separable in $y$.

Let $a_1,\ldots,a_{n-3}$ be integers such that $g(a,b,c,y)$ is separable in $y$, and recall \eqref{fdefabc}. By \cite[Theorem 1]{Coh1980}, applied to the field $F = \overline{\bQ}(a)$, we have 
\[
\Gal(f, \overline{\bQ}(a,b,c)) = S_n.
\]
Suppose for a contradiction that
\[
g(a,b,c,y) = g_1(y;a,b,c) g_2(y;a,b,c) \in \overline{\bQ}(a,b,c)[y],
\]
for some non-constant polynomials $g_1, g_2$. Let $r_{\sig_i}$ be a root of $g_i$ ($i=1,2$), where for $\sig \in S_n$ we write
\[
r_\sig = \sum_{\tau \in G} \prod_{i \le n} \alp_{\sig \tau (i)}^i,
\]
where $\alp_1, \ldots, \alp_n$ are the roots of $f$ in $\overline{\bQ(a,b,c)}$. Then for $\kap = \sig_2 \sig_1^{-1} \in \Gal(f, \overline{\bQ}(a,b,c))$ we have $\kap(r_{\sig_1}) = r_{\sig_2}$. As $g_1$ has coefficients in $\overline{\bQ}(a,b,c)$, we now see that $g_1$ and $g_2$ have a common root $\kap(r_{\sig_1}) = r_{\sig_2}$, which is impossible because $g$ is separable. 

We conclude that there are $O(H^{n-4})$ integer tuples 
\[
(a_1, \ldots, a_{n-3}) \in [-H,H]^{n-3}
\]
such that $\cY_{a_1, \ldots, a_{n-3}}$ is reducible over $\overline{\bQ}$. The upshot is that there are $O(H^{n-4})$ degenerate tuples $(a_1,\ldots,a_{n-3}) \in [-H,H]^{n-3}$ in total. Now let us fix such a tuple.

Write
\[
\fF(a,b,c,X)
= X^n + a_1X^{n-1} + \cdots + a_{n-3}X^3 + aX^2 + bX + c.
\]
By \cite[Lemma 2]{Die2012}, there are at most $O(H)$ values of $a,b \in \bZ \cap [-H,H]$ for which the polynomial $\fF(a,b,c,X) \in \bQ(c)[X]$ has non-$S_n$ Galois group, where the implied constant is uniform in $a_1,\ldots,a_{n-3}$. The total contribution from these $O(H)$ special choices of $a,b$ is $O(H^{n-2})$, since there are $O(H)$ possibilities for $c$.

For the other $O(H^2)$ specialisations of $a,b$, the polynomial
\[
\cF(c,X) := \fF(a,b,c,X) \in \bQ(c)[X]
\]
is separable with Galois group $S_n$. By Lemma \ref{GeneralResolvent}, the polynomial 
\[
h(c,y) := \Phi_{\cF,G}(c,y) \in \bZ[c,y]
\]
is irreducible, so by Theorem \ref{LopsidedBP} there are $O(H^{\eps+1/d})$ pairs $(c,y) \in \bZ^2$ such that 
\[
h(c,y) =0,
\qquad
|c| \le H, \qquad y \ll_n H^{O_n(1)}.
\]
We conclude that there are $O(H^{\eps+n-2+1/d})$ tuples $(a_1,\ldots,a_n) \in (\bZ \cap [-H,H])^n$, with $(a_1,\ldots,a_{n-3})$ degenerate, such that $G_f=G$.

\bigskip

In total there are $O(H^{\eps + n- 3/2+ 3d^{-1/3}})$ irreducible polynomials \eqref{fdefGeneral} with $\max_i |a_i| \le H$ and Galois group $G$.

\section{Even permutation groups}
\label{EvenGroups}

In this section, we establish Theorem \ref{EvenGroupThm}.
The inequality $d \ge 6$ follows from Theorem \ref{HigherThm} when $n \ge 11$, and from the classification of transitive groups of low degree \cite{BM1983} when $3 \le n \le 10$. Let $(a_1,\ldots,a_{n-2}) \in (\bZ \cap [-H,H])^{n-2}$, and let
\[
g(a,b,y) =
\Phi(y; a_1,\ldots,a_{n-2},a,b).
\]

First suppose $(a_1,\ldots,a_{n-2})$ is non-degenerate, meaning that $g$ is absolutely irreducible. Denote by $\cY$ the affine surface cut out by the vanishing of $g$. By Theorem \ref{BrowningThm}, there exist $g_1,\ldots,g_J \in \bZ[a,b,y]$ with $J \ll H^{\eps+1/\sqrt d}$, and a finite set $\cZ \subset \cY$, such that
\begin{enumerate}
\item Each $g_j$ is coprime to $g$ and has degree $O(1)$;
\item $|\cZ| \ll H^{\eps+2/\sqrt d}$;
\item If $(a,b,y) \in \cY \cap \bZ^3 \setminus \cZ$ and
\[
|a|,|b| \le H, \qquad y \ll_n H^{O_n(1)}
\]
then 
\begin{equation} \label{curve2}
g(a,b,y) = g_j(a,b,y) = 0
\end{equation}
for some $j$.
\end{enumerate}
The contribution to $N_{G,n}$ from $(a,b,y) \in \cZ$ is $O(H^{\eps+n-2+2/\sqrt d}) \ll H^{\eps+n-3/2+1/\sqrt d}$, so in the non-degenerate case it remains to consider \eqref{curve2} for $j$ fixed.

If $\deg_y(g_j) = 0$ then let $F(a,b) = g_j(a,b)$, and otherwise let $F(a,b)$ be the resultant of $g$ and $g_j$ in the variable $y$. Then $F(a,b) = 0$ whenever we have \eqref{curve2}, whereupon $\cF(a,b) = 0$ for some irreducible divisor $\cF(a,b) \in \bQ[a,b]$ of $F$. If $\cF$ is non-linear, then Theorem \ref{LopsidedBP} yields
\[
\# \{ (a,b) \in (\bZ \cap [-H,H])^2: \cF(a,b) = 0 \} \ll H^{\eps+1/2},
\]
and the contribution to $N_{G,n}$ from this case is $O(H^{\eps + n-2 + 1/\sqrt d + 1/2}) = O(H^{\eps+n-3/2+1/\sqrt d})$. In the non-degenerate case, it remains to treat the scenario in which $\cF$ is linear. 

If the $a$ coefficient of $\cF(a,b)$ is non-zero, then we have $a = c_1 b + c_2$ for some $c_1, c_2 \in \bQ$. As $G < A_n$, we have
\[
P(b,z) := z^2 - \Del(a_1,\ldots,a_{n-2}, c_1 b + c_2, b) = 0
\]
for some $z \in \bN$, where $\Del(a_1,\ldots,a_n)$ denotes the discriminant of $X^n + a_1 X^{n-1} + \cdots + a_n$. The polynomial $P$ is irreducible, by \cite[Lemma 5]{Die2013}, see \cite[\S 6]{DOS} for why we assume \eqref{erratum}. Let $C>0$ be large, and note that $\Del$ has total degree $2n-2$, so if 
$|a_1|,\ldots,|a_n| \le H$ then $|\Del(a_1,\ldots,a_n)| \le CH^{2n-2}$. Now Theorem \ref{LopsidedBP} yields
\[
\# \{ (b,z) \in \bZ^2: |b| \le H, |z| \le CH^{n-1}, P(b,z) = 0 \} \ll H^{\eps+1/2},
\]
and the contribution to $N_{G,n}$ from this case is $O(H^{\eps+n-3/2+1/\sqrt d})$.

Otherwise $\cF(a,b) = \lam(b-\mu)$, for some $\mu \in \bQ$ and some $\lam \in \bQ \setminus \{0 \}$, and
\[
Q(a,z) := z^2 - \Del(a_1,\ldots,a_{n-2},a,\mu) = 0
\]
for some $z \in \bN$. The polynomial $Q$ is irreducible, by \cite[Lemma 6]{Die2013}, see \cite[\S 6]{DOS} for why we assume \eqref{erratum}. Now Theorem \ref{LopsidedBP} yields
\[
\# \{ (a,z) \in \bZ^2: |a| \le H, |z| \le CH^{n-1}, Q(a,z) = 0 \} \ll H^{\eps+1/2},
\]
and the contribution to $N_{G,n}$ from this case is $O(H^{\eps+n-3/2+1/\sqrt d})$.

\bigskip

For degenerate tuples, we can follow \S \ref{DegenerateTuples} with minimal changes. We find that their contribution to $N_{G,n}$ is
$O(H^{\eps+n-2+1/d})$.

We have considered all cases, completing the proof of Theorem \ref{EvenGroupThm}.

\providecommand{\bysame}{\leavevmode\hbox to3em{\hrulefill}\thinspace}

\end{document}